\documentclass[11pt]{amsart}

\usepackage{amsmath}
\usepackage{amsfonts}
\usepackage{amsthm}
\usepackage{amssymb}
\usepackage{hyperref}
\usepackage{graphicx}
\usepackage{tikz-cd}
\usepackage{multicol}
\usepackage{enumitem}
\usepackage{indentfirst}
\usepackage[none]{hyphenat}
\usepackage{xcolor}
\usepackage{mathrsfs}
\usepackage{cleveref}
\usepackage{cancel}

\usepackage[top=4cm, bottom=4cm, left=3cm, right=3cm]{geometry}

\binoppenalty=\maxdimen
\relpenalty=\maxdimen

\usepackage{lmodern}

\newcommand{\RR}{\mathbb{R}}

\newcommand{\sm}{\setminus}

\newcommand{\Ter}[1]{\mathfrak{ter} (#1)}
\newcommand{\lset}[1]{\mathfrak l (#1)}
\newcommand{\rset}[1]{\mathfrak r (#1)}
\newcommand{\tset}[1]{\mathfrak t (#1)}
\newcommand{\bset}[1]{\mathfrak b (#1)}
\newcommand{\boxset}[1]{\mathfrak{box} (#1)}
\newcommand{\wset}[1]{\mathfrak w (#1)}
\newcommand{\hset}[1]{\mathfrak h (#1)}

\DeclareMathOperator{\nextF}{\mathfrak{S_F}}

\DeclareMathOperator{\desc}{\prec}

\newcommand{\adj}{\curvearrowright}

\DeclareMathOperator{\Rl}{R}
\DeclareMathOperator{\notRl}{\cancel{\Rl}}

\newcommand{\im}[1]{im(#1)}

\newtheorem{theorem}{Theorem}
\newtheorem{lemma}[theorem]{Lemma}
\newtheorem*{mainthm*}{Main Theorem}
\newtheorem{corollary}[theorem]{Corollary}
\newtheorem{remark}[theorem]{Remark}

\newtheorem{definition}[theorem]{Definition}
\newtheorem{statement}{Statement}

\newtheorem{property}[theorem]{Property}

\graphicspath{ {figures/} }

\title{Burling graphs as intersection graphs}
\author{Pegah Pournajafi}
\thanks{ENS Lyon, LIP, University Lyon 1, UCBL, CNRS, Lyon, France}

\begin{document}

\maketitle

\begin{abstract}

For a subset $ S $ of $ \RR^d$, $ S$-graphs are the intersection graphs of specific transformations of $ S $. The class of Burling graphs is a class of triangle-free graphs with arbitrarily large chromatic number that has attracted much attention in the last years. In 2012, Pawlik, Kozik, Krawczyk, Laso\'n, Micek, Trotter, and Walczak showed that for every compact and path-connected set $ S \subseteq \RR^2$ that is different from an axis-parallel rectangle, the class of $ S $-graphs contains all Burling graphs. There is, however, a gap between the two classes. In the recent years, there have been improvements in understanding the subclasses of $ S $-graphs that are closer or equal to Burling graphs. In this article, we close this gap for every set $ S $ with the mentioned properties: we introduce the class of constrained $ S $-graphs, a subclass of $ S$-graphs, and prove that it is equal to the class of Burling graphs. We also introduce the class of constrained graphs, a subclass of intersection graphs of subsets of $ \RR^2$, and prove that it is equal to the class of Burling graphs. 

\end{abstract}

\section{Introduction} \label{sec:intro}

\subsection*{Intersection graphs of geometric objects}

Let $ \mathcal F $ be a family of sets. The \emph{intersection graph} of $ \mathcal F $ is the graph $ G $ where $V(G) = \mathcal F $ and $ E(G) = \{ST \mid S \neq T, S \cap T \neq \varnothing \}$. A \emph{geometrical object}, in this setting, is a subset of a Euclidean space~$\RR^d$.

In this article, we deal only with the following type of \emph{transformations} of $ \RR^d$: transformations $ T: \RR^d \rightarrow \RR^d $ of the form  $ T = (T_1, \dots, T_d) $ where each $T_i $ is an affine function from $ \RR $ to itself. In particular, in case $d=2 $, we have 
$
T(x,y) = (ax+c, by+d), 
$
for some $ a, b \in \RR^* = \RR \sm \{0\} $ and $ c,d \in \RR$.
From now on, every time that we talk about a transformation, we mean a transformation of the form above. 

For a set $ S $, we call every set of the form $ T(S) $, where $ T $ is a transformation, a \emph{transformed copy} of $ S $. So a transformed copy of $ S $ is a set obtained from a translation of $ S $ and independent scalings parallel to the axis. 

We say that a graph $ G $ is an \emph{$ S $-graphs} if it is the intersection graph of a family $ \mathcal F $ of geometric objects that are transformed copies of $ S $.

\subsection*{Chromatic number of intersection graphs}

The \emph{chromatic number} of a graph $ G $, denoted by $ \chi(G)$, is the smallest integer $ k $ such that the vertex-set of $ G $ can be partitioned into $ k $ stable sets. The \emph{clique number} of $ G $, denoted by $ \omega(G) $ is the maximum number of vertices in $ G $ that are two by two adjacent. It is easy to observe that $\chi(G) \geq \omega(G)$.

A class of graphs is said to be \emph{hereditary} if it is closed under taking induced subgraphs. A hereditary class $\mathcal C $ of graphs is said to be $\chi$-bounded if there exists a function $ f: \mathbb N \rightarrow \mathbb N $ such that for all $ G \in \mathcal C $, we have $ \chi(G) \leq f(\omega(G))$. In particular, if the triangle-free graphs (that is, graphs with clique number at most 2) in $ \mathcal C $ have arbitrarily large chromatic number, then $ \mathcal C $ is not $chi$-bounded. In this article, we are interested in the special case where these classes of graphs under study are intersection graphs of geometric objects, and in particular, $ S $-graphs.

Let us see some examples. If $ S $ is an interval in $\RR$, then the class of $ S $-graphs is known as \emph{interval graphs}. It is well-known that interval graphs are perfect graphs (see \cite{Trotignon2013survey}) and thus interval graphs are $\chi$-bounded. 

In \cite{Asplund1960}, Asplund and Gr{\"u}nbaum generalized this result to $\RR^2$: the class of intersection graphs of axis-parallel rectangles in~$ \RR^2 $ is $\chi$-bounded. 

Starting from the third dimension, however, the situation changes.
In 1965, in \cite{Burling1965}, Burling proved that the class of intersection graphs of axis-parallel boxes (cuboids) in $ \RR^n$ is not $\chi$-bounded when $ n \geq 3 $. The core of his proof is to first reduce the problem to the case $n=3$ and then define a sequence of triangle-free graphs with arbitrarily large chromatic number, known as the \emph{Burling sequence}, where each graph in the sequence is the intersection graphs of axis-parallel boxes in $\RR^3$.
The class generated by the Burling sequence, i.e.\ the class of all induced subgraphs of the graphs in the Burling sequence, is known as the class of \emph{Burling graphs}. We describe this class via an equivalent definition later in the introduction. 

In 1970s, Erd\H{o}s asked whether the class of intersection graphs of line segments in $ \RR^2$ is a $ \chi$-bounded class. In 2012, Pawlik, Kozik, Krawczyk, Laso\'n, Micek, Trotter, and Walczak \cite{Pawlik2014linesegment} answered negatively to this question by proving that the class of line segment graphs contains all graphs of the Burling sequence. 

Later, the same authors \cite{pawlik2013general} generalized this result to any other set $ S \subseteq \RR^2 $ that has some reasonable constraints. They prove that for every compact and path connected set $ S \subseteq \RR^2 $ different from an axis-parallel rectangle (we call such sets \emph{Pouna sets}), the class of $ S $-graphs contains the Burling sequence and therefore is not $ \chi$-bounded. It is worth noting that this result, along with the earlier mentioned result of Asplund and Gr{\"u}nbaum, completes the study of $\chi$-boundedness of $ S $-graphs for compact and path-connected subsets of $ \RR^2 $. 

In \cite{Pawlik2014linesegment}, it is also explained how this result disproves a conjecture of Scott (Conjecture 8 in \cite{Scott1997}) from 1997. This new application of Burling graphs created new motivations to know this class of graphs better, in particular as intersection graphs. In particular, the question of finding proper subclasses of $ S $-graphs are equal to or are closer to the class of Burling graphs. 

With this motivation, in 2016, Chalopin, Esperet, Li and Ossona de Mendez \cite{Chalopin2016} studied the class of \emph{restricted frame graphs}, a class first introduced in \cite{Krawczyk2014}. A \emph{frame} is the boundary of an axis-parallel rectangle in $ \RR^2 $ with non-empty interior. From \cite{Pawlik2014linesegment}, we know that frame graphs (i.e.\ $ S $-graphs where $ S $ is a frame) contain Burling graphs. Restricted frame graphs are defined by setting a few restriction on the interaction of frames and as a result of those restrictions form a proper subclass of frame graphs. The class of restricted frame graphs, however, still contains all graphs in the Burling sequence and thus is not $\chi$-bounded. 

In~\cite{BG1PournajafiTrotignon2021}, Trotignon and the author introduced the class of \emph{strict frame graphs}, a subclass of restricted frame graphs, by adding one more restriction to the set of restrictions defined in \cite{Chalopin2016}. We proved that the class of strict frame graphs is the smallest subclass of frame graphs that contains all graphs in the Burling sequence. Similarly, by setting a few restriction on how the sets can intersect, we defined \emph{strict line-segment graphs} and \emph{strict box graphs}, the smallest subclasses of line segment graphs and box graphs, respectively, that contain all graphs in the Burling sequence.

\subsection*{In this article} 

For any Pouna set $S$ (that is, a compact and path-connected subset of $ \RR^2 $ that is different from an axis-parallel rectangle), by setting constraints on how the sets can interact, we define the class of \emph{constrained $ S $-graphs} and prove that the class of constrained $ S $-graphs is the smallest subclass of $ S$-graphs containing all graphs in the Burling sequence. In other words, for any such set $ S $, the class of constrained $ S $-graphs is equal to the class of Burling graphs. We also prove that the classes of constrained $S $-graphs for different sets $ S$ are all equal to a class that we call \emph{constrained graphs}.

Even though some terms are not defined here, we state a simplified form of our main theorem below. The same theorem holds for oriented graphs or non-oriented graphs, as we will discuss later. For the precise
statements, see \Cref{thm:main-complete-theorem} in \Cref{sec:proof-of-equality}.
\begin{mainthm*}
	Let $ G $ be a graph. The followings are equivalent:
	\begin{enumerate}
		\item $ G $ is a constrained graph, i.e.\ the intersection graph of a finite family of Pouna sets satisfying constraints \textup{(C1)-(C5)}. 
		\item $ G $ is an abstract Burling graphs and equivalently, a Burling graph.
		\item $ G $ is a constrained $ S $ graphs for some Pouna set $ S $, i.e.\ it is an $ S $-graph satisfying some constraints \textup{(C1)-(C6)}.
	\end{enumerate}
\end{mainthm*}

A core idea in the proof of the theorem above is to use abstract Burling graphs, an equivalent definition of Burling graphs by Trotignon and the author from \cite{BG1PournajafiTrotignon2021}. The generality of this definition allows for a very short proof of $(1)\implies (2)$. 

Let us state this definition already here. 

\begin{definition} \label{def:Burling-set}
	A \emph{Burling set} 
	is a triple $ (S, \desc, \adj) $ where $ S $ is a non-empty set, $\desc$ is a strict partial order on $S$,
	$ \adj $ is a binary relation on $S$ that does not have directed cycles, and such that the following axioms hold:
	
	\begin{enumerate}
		\item[\textbf{\textup{(A1)}}] if $ x \desc y $ and $ x \desc z $, then either $ y \desc z $ or $ z \desc y $, \label{item:descdesc}
		\item[\textbf{\textup{(A2)}}] if $ x \adj y $ and $ x \adj z $, then either $ y \desc z $ or $ z \desc y $, \label{item:adjadj}
		\item[\textbf{\textup{(A3)}}] if $ x \adj y $ and $ x \desc z $, then $ y \desc z $, \label{item:adjdesc}
		\item[\textbf{\textup{(A4)}}] if $ x \adj y $ and $ y \desc z $, then either $ x \adj z $ or $ x \desc z $. \label{item:transitiveboth}
	\end{enumerate}
\end{definition}

Notice that the tuple $ (S, \adj) $ is an oriented graph. 

\begin{definition} \label{def:abstract-Burling-graphs}
	An oriented graph $ G $ is an \emph{abstract Burling graph}\index{abstract Burling graphs} if there exists a partial order~$ \prec $ on~$ V(G) $ such that $ (V(G), \prec, A) $ forms a Burling set. A \emph{non-oriented abstract Burling graph} is the underlying graph of an oriented abstract Burling graph.
\end{definition}

Notice that if $ (S, \prec, \adj) $ is a Burling set and $ G = (S, \adj) $ is its corresponding abstract Burling graph, then for every induced subgraph $  G' $ of $ G $ is of the form $ (S', \adj)$ for some $ S' \subseteq S $. Moreover $ S'$ itself forms a Burling set with inherited relations $ \desc $ and $ \adj $. So, the set of all (oriented or non-oriented) abstract Burling graphs forms a class of graphs.

\begin{theorem}[Pournajafi and Trotignon \cite{BG1PournajafiTrotignon2021}] \label{thm:PTabstract-BG}
	A graph $ G $ is an abstract Burling graph if and only if it is a Burling graph. 
\end{theorem}

\subsection*{Paper outline} In \Cref{sec:notation}, we introduce the notations needed for this article. In \Cref{sec:constrained-S-graphs}, we introduce the classes of constrained $ S $-graphs as well as constrained graphs.  And then, in \Cref{sec:proof-of-equality}, we state the main theorem in details and prove that the classes of constrained graphs, constrained $ S $-graphs for any Pouna set $ S$, and Burling graphs are all equal.
To increase the readability of the paper, we include the proof of some basic topological lemmas in \Cref{appendix-proofs}. 

\section{Notation} \label{sec:notation}

The notation for graphs and oriented graphs are the standard definitions in graph theory. For any graph theoretical notion not defined here, we refer to~\cite{BondyMurty}.

We say that a hereditary class $ \mathcal C $ of graphs is \emph{generated} by a set $ \mathcal H $ of graphs if $ \mathcal C $ is exactly the class of all graphs $ G $ such that $ G $ is an induced subgraph of some graph in~$ \mathcal H $.

Let $ S $ be a set, and let $ \Rl $ be a binary relation on $ S $, that is, $\Rl \subseteq S\times S$. We write $ x \Rl y $ for $ (x,y) \in \Rl $, and $ x \notRl y $ for $ (x,y) \notin \Rl $. For an element $ s\in S $, we denote by $ [s \Rl] $ the set $ \{t \in S : s \Rl t \} $.
A \emph{directed cycle} in $ \Rl $ is a set of elements $ x_1, x_2, \dots, x_n $, with $ n \in \mathbb{N} $, such that $ x_1\Rl x_2$, $ x_2 \Rl x_3 $, $\dots$, $ x_n \Rl x_1 $. Note that when we deal with relations, we allow cycles on one or two elements. So, strict partial orders do not have directed cycles. In fact, a relation $\Rl $ has no directed cycles if and only if its transitive closure is a strict partial order.

For any topological notion not defined here, we refer to~\cite{MunkresTopology}. We always consider $ \RR^n $ with its usual topology. As explained in \Cref{sec:intro}, in this article, we only consider transformations of $ \RR^n $ whose projection on each axis is an affine function.  

For a set $ S $ in $\RR^d$, we denote the \emph{interior}\index{interior} and the \emph{closure}\index{closure} of $ S $ respectively by $ S^\circ$ and $ \bar{S}$. Moreover, we denote the \emph{boundary}\index{boundary} of $ S $ by $ \partial S $, i.e. $ \partial S = \bar{S} \sm S^\circ $. We denote the ball of radius~$ r $ and center~$ c $ in $ \RR^d $ by $ D(c, r) $. For a function $ f $, we denote its image by $\im{f} $, and its restriction to a set $ A $ in its domain by $ f|_{A}$.  We denote the projection on the $ i$-th axis in $ \RR^n $ by $ \rho_i$. 

We say that the path $ \gamma : [0,1] \rightarrow \RR^d $ in $\RR^d $ \emph{joins} the two points $ \gamma(0)$ and $ \gamma(1)$. Two paths $ \gamma_1:[0,1]\rightarrow$ and $ \gamma_2: [0,1]\rightarrow \RR^n $ are said to be \emph{internally disjoint}\index{path!internally disjoint} if
$$   \gamma_1([0,1]) \cap \gamma_2([0,1]) \subseteq \{\gamma_1(0), \gamma_1(1)\} \cap \{\gamma_2(0), \gamma_2(1)\} .$$

A \emph{box}\index{box} in $\RR^d $, is a set of the form $ B = \prod_{i=1}^{d} I_i, $ where $ I_i $ is a closed interval (thus possibly empty) in $ \RR $. So, boxes in $ \RR $ are intervals, in $ \RR^2$ are axis-parallel rectangles, and in $ \RR^3$ are axis-parallel cuboids. A \emph{frame}\index{frame} is the boundary of a box with non-empty interior in $ \RR^2 $.

Now let us focus on $ \RR^2$. Let $ S $ be a bounded subset of $ \RR^2 $. We define the following notions on $ S $: \begin{align*} \lset S &= \inf \{x : \exists y \ (x,y) \in S \}, \\ \rset S &= \sup \{x : \exists y \ (x,y) \in S \}, \\ \bset S &= \inf \{y : \exists x \ (x,y) \in S \}, \\ \tset S &= \sup \{y : \exists x \ (x,y) \in S \}.  \end{align*} 
The letters $\mathfrak{l} $, $ \mathfrak{r} $, $\mathfrak{b}$, and $ \mathfrak{t} $ stand for \emph{left}, \emph{right}, \emph{bottom}, and \emph{top}, respectively. If $ S $ is a compact set in $ \RR^2 $, then all the values above are finite and also, we can replace  $\inf$ and $\sup$ by $ \min $ and $\max$ respectively. In this case, we also define $ \wset S = \rset S - \lset S$ and $ \hset{S} = \tset{S} - \bset{S} $. The letters $ \mathfrak{w} $ and $\mathfrak{h} $ stand for \emph{width} and \emph{height} respectively. Notice that if $ S' \subseteq S $, we have $ \lset{S'} \geq \lset{S} $, $ \rset{S'} \leq \rset{S} $, $ \bset{S'}\geq \bset{S} $, and $ \tset{S'} \leq \tset{S}$. 

The \emph{bounding box}\index{bounding box} of a bounded set $ S \subseteq \RR^2 $, denoted by $ \boxset S $, is the (inclusion-wise) smallest closed rectangle in $\RR^2$ containing $ S $. Equivalently,
$$ \boxset S  = [\lset S, \rset S] \times [\bset S, \tset S]. $$ 
So, $ \lset{\boxset{S}} = \lset{S} $, $ \rset{\boxset{S}} = \rset{S} $, etc.  If $ \mathcal F $ is a family of bounded subsets of $ \RR^2 $ and $ T $ is a transformations (of the form mentioned earlier), we use the unconventional notation $ T(\mathcal F) $ for the family $ \{T(S) : S \in \mathcal F \} $. It is easy to see that that $ \boxset{T(\mathcal F)} = T(\boxset{\mathcal F})$.

Recall that with the mentioned constraint on transformations of $ \RR^n$, any transformation $ T:\RR^2 \rightarrow \RR^2 $ of $ \RR^2 $ that we deal with is of the form 
$$ T(x,y) = (ax+c, by+d),  $$  for some $ a, b \in \RR^* = \RR \sm \{0\} $ and $ c,d \in \RR$.
We say that $ T $ is a \emph{positive}\index{transformation!positive} transformation if $ a > 0 $ and $ b >0 $. It is easy to see that positive transformations with composition form a group. In particular: \begin{itemize} \item the composition of two positive transformations is a positive transformation,  \item every positive transformation has an inverse. \end{itemize}

Several times, we use the fact that if $ T:(x,y) \mapsto (ax+c, by+d) $ is a positive transformation and $ S $ is a compact set in $ \RR^2$, then setting $ S' = T(S) $, we have: \begin{multline*} \lset{S'} = a. \lset{S} + c, \  \rset{S'} = a. \rset{S} +c, \ \bset{S'} = b. \bset{S}+d, \ \text{and } \tset{S'} = b. \tset{S}+d.  \end{multline*} 
In particular, $ \boxset{T(S)} = T(\boxset{S})$. 

We say that $ S' $ is a \emph{positive transformed copy}\index{transformed copy!positive} of $ S $ if $ S'=T(S) $ for some positive transformation $ T$. The \emph{horizontal reflection} of $ S $ is $T(S) $ where~$ T $ is the transformation that maps~$(x,y)$ to~$(-x, y)$.

\section{Constrained graphs and constrained $ S $-graphs} \label{sec:constrained-S-graphs}

In this section we define the class of contained graphs and the class of constrained $ S$-graphs. But first, we need some definitions and lemmas.

\subsection{Pouna sets and their territories}

\begin{definition}\label{def:PounaSet}
	A subset $ S $ of $ \RR^2 $ is said to be a \emph{Pouna set}\index{Pouna sets} if it is path-connected and compact, and is not an axis-parallel rectangle.
\end{definition}

	The \emph{territory} of a Pouna set $ S $, denoted by $\Ter S $, is defined as follows:
	\begin{equation*}
		\Ter S = \{(x,y)  \in \boxset S \sm S : \exists x'\in \RR \text{ s.t. } x' > x \text{ and } (x',y) \in S \}.
	\end{equation*}

	We say that a Pouna set $ S $ is \emph{strong} if it has a non-empty territory.

In Figure~\ref{fig:territory_examples}, some examples of strong Pouna sets and their territories are represented. In particular, a frame is an example of a strong Pouna set.

In figures of Pouna sets, we do not always represent the territory as it is well-defined given the Pouna set. But whenever we represent the territory, we show the Pouna set in solid colors, and the territory in hatch.

\begin{figure} 
	\centering 
	\vspace*{-1cm} 
	\includegraphics[width=9.5cm]{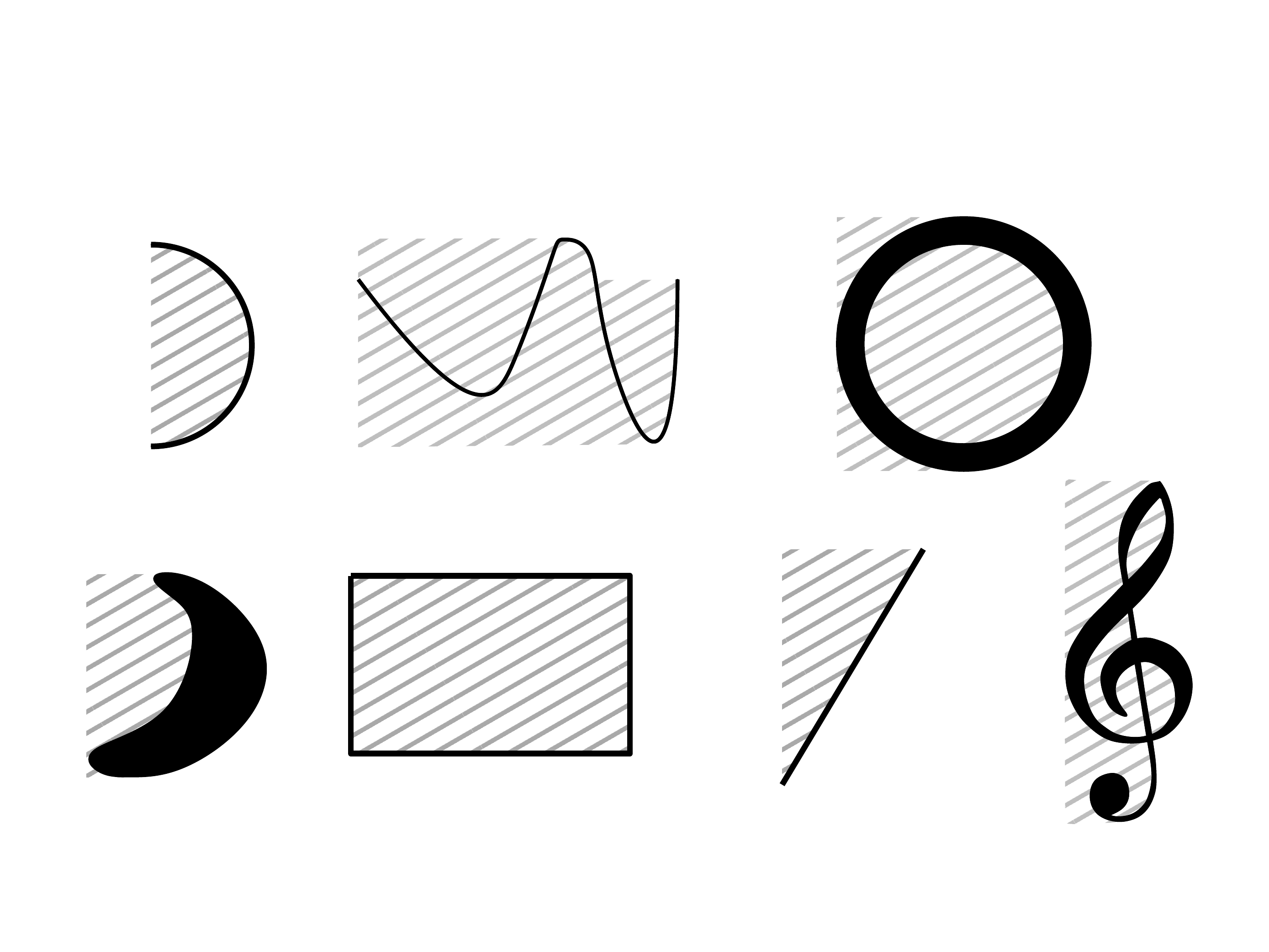} 
	\vspace*{-1cm} 
	\caption{Examples of strong Pouna sets and their territories. The Pouna sets are shown in black and their territories in hatch.} \label{fig:territory_examples} 
\end{figure}

Strong Pouna sets give us the possibility of using the properties of the territory, and they are not more restrictive than general Pouna sets as \Cref{prop:strong-perkins-S-or-horizontal-reflection} below shows. 

We first need a lemma whose proof can be found in \Cref{appendix-proofs}. 

\begin{lemma} \label{lem:real-border-lemma} Let $ X $ be a topological space and let $ A, B \subseteq X$. If $ B $ is connected, $ B \cap A^\circ \neq \varnothing$, and $ B \cap [X \sm \bar{A}] \neq \varnothing $, then $ B \cap \partial A \neq \varnothing $.  
\end{lemma}

\begin{property} \label{prop:int-box-min-S-non-empty} 
	If $ S $ a Pouna set, then $ \boxset{S}^\circ \sm S \neq \varnothing$. 
\end{property} 
\begin{proof} First of all, $ S $ is not a subset of an axis-aligned  line-segment. So, the closure of $\boxset{S}^\circ $ is equal to $\boxset{S}$. Now, if $\boxset{S}^\circ \sm S = \varnothing $, then $ \boxset{S}^\circ \subseteq S \subseteq \boxset{S} $, and since $ S $ is closed, we have $ S = \boxset{S} $, and $ S $ is an axis-aligned rectangle.  
\end{proof}

\begin{lemma} \label{prop:strong-perkins-S-or-horizontal-reflection} For every Pouna set $ S $, either $ S $ or its horizontal reflection is strong. 
\end{lemma} 

\begin{proof} Let $ S' = T(S) $ be the horizontal reflection of $ S $ (thus, $ T:(x,y) \mapsto (-x,y)$). 
	
	By \Cref{prop:int-box-min-S-non-empty}, we can choose a point $ p = (x,y) \in \boxset{S}^\circ \sm S$. Let $ L $ be the horizontal line passing through $ p $, and set $ A $ to be the closed half-plane consisting of the points on $ L $ and under $L$. Notice that $ \bset{S} < y < \tset{S} $, so $ S $ has a point on the top-side of $ \boxset{S}$, thus outside $A = \bar{A} $ and a point on the bottom-side of $\boxset{S}$, thus inside $ A^\circ $. Setting $ B = S $ in the statement of \Cref{lem:real-border-lemma}, we conclude that $ S \cap L \neq \varnothing $. In other words, there is a point $ p= (x',y) \in \mathcal S $.   If $ x'> x $, then $ p \in \Ter{S} $, and $ S $ is strong. If $ x'< x $, then $ -x' > -x $. Notice that $ (-x', y) \in S' $ and $ (-x, y) \in \boxset{S'}\sm S' $. So, $(-x, y) \in \Ter{S'} $, and $ S' $ is strong. 
\end{proof}

Let $ A $ and $ B $ be two strong Pouna sets. We write \emph{$A \prec B$} if $ \boxset{A} \subseteq \Ter{B}$.

\begin{figure}
	\centering
	\includegraphics[width=10cm]{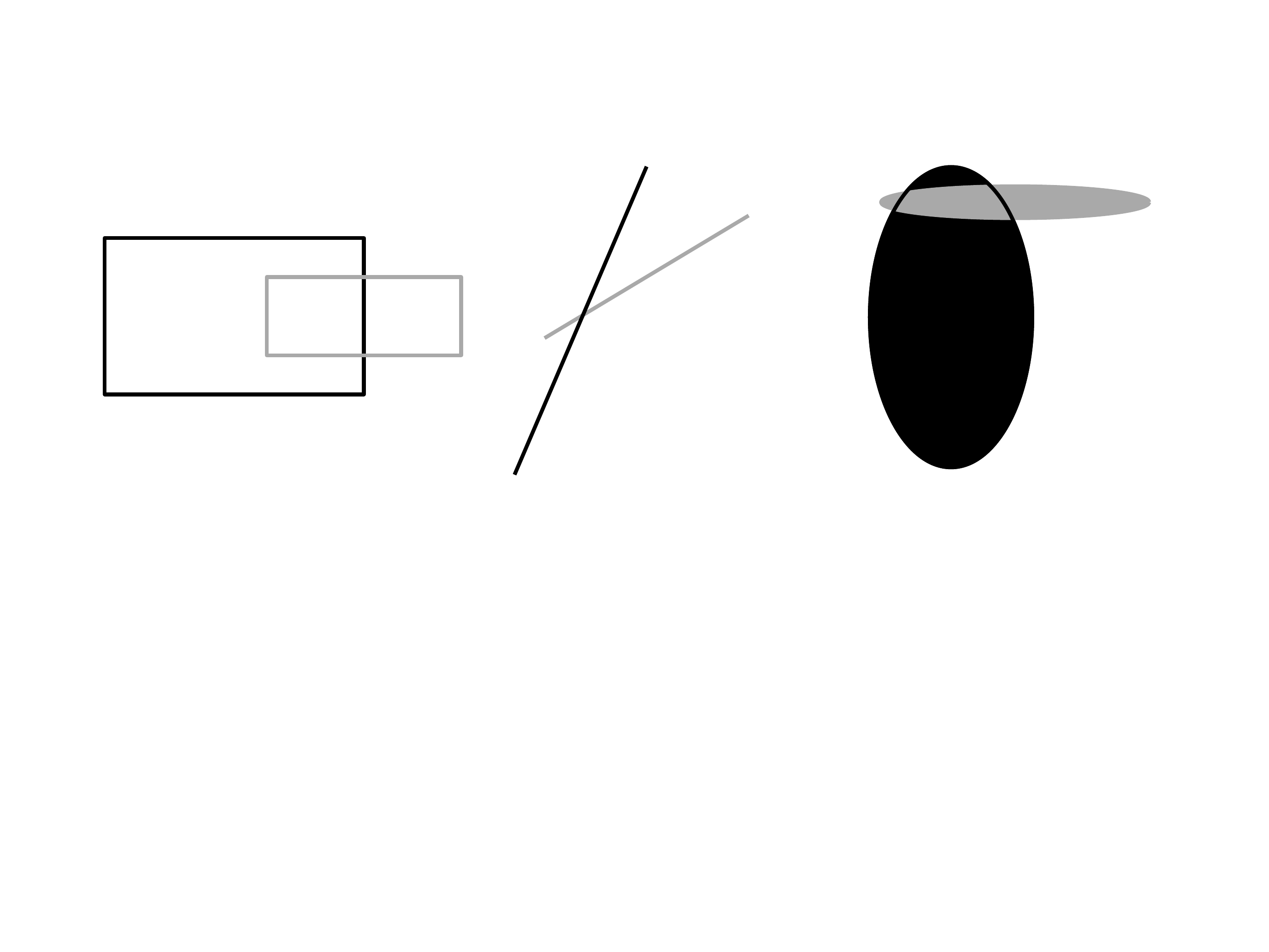}
	\vspace*{-3cm}
	\caption{The relation $ A \adj B $. In the figures above, $ A $ is shown in gray and $ B $ in black.} \label{fig:adj-examples}
\end{figure}

We also write \emph{$ A \adj B $} if $ A $ and $ B $ are distinct intersecting sets with the following properties: 
\begin{itemize} 
	\item $\lset B \leq \lset A < \rset B < \rset A $, 
	\item $\bset B < \bset A < \tset A < \tset B$, 
	\item $ \{ (x,y) \in A : x  = \lset A \} \subseteq \Ter B $.  
\end{itemize} 
See \Cref{fig:adj-examples} for some examples.

\subsection{Constrained graphs} \label{subsec:constrained-graphs}

Let $ \mathcal F $ be a non-empty and finite family of strong Pouna sets satisfying the following constraints: 
\begin{enumerate}	 
	\item[\textbf{\textup{(C1)}}]  for every $ A, B \in \mathcal F $, if $ A \neq B $ and $ A \cap B \neq \varnothing $, then,  either $ A \adj B $ or $ B \adj A $. \label{gcond:adjacent}
	
	\item[\textbf{\textup{(C2)}}] For every $ A, B \in \mathcal F $, if $ A \cap B = \varnothing $ and  $ A \cap \Ter B  \neq \varnothing$,  then  $ A \prec B$. \label{gcond:prec}
	
	\item[\textbf{\textup{(C3)}}] For every $ A, B \in \mathcal F $, if $ A \neq B $ and $ A \cap B \neq \varnothing $, then there exists no $ C \in \mathcal F $ such that $ C \subseteq \Ter A  \cap \Ter B $. \label{gcond:common-territory}
	
	\item[\textbf{\textup{(C4)}}] There exist no $ A, B, C \in \mathcal F $ such that $ A \prec B $, $ A \adj C $, and $ B \adj C $.  \label{gcond:strict-condition}
	
	\item[\textbf{\textup{(C5)}}] The maximum number of pairwise intersecting and distinct elements in $ \mathcal F $ is at most two. \label{gcon:triangle-free} 
\end{enumerate}

Let $ \mathcal F $ be a finite family of strong Pouna sets satisfying \textup{(C1)}. We say that $ G $ is the \emph{oriented intersection graph} \index{intersection graph!oriented} of $ \mathcal F $ if $ V(G) = \mathcal F$ and $  A(G) = \{AB: A \adj B \} $. Notice that the underlying graph of $ G $ is the intersection graph of $\mathcal F$ because for distinct element $ A, B \in \mathcal F $, we have $ A\cap B \neq \varnothing $ if and only if $ A \adj B $ or $ B \adj A$.

\begin{definition} \label{def:constrained-graphs}\index{constrained graphs}
	An oriented graph (resp.\ graph) is called an \emph{oriented constrained graph} (resp.\ a \emph{constrained graph}) if it is isomorphic to the oriented intersection graph (resp.\ intersection graph) of a non-empty family of strong Pouna sets satisfying Constraints \textup{(C1)-(C5)}.
\end{definition}

\subsection{Constrained $ S$-graphs}

\begin{definition} \label{def:constrained-S-graphs}\index{constrained $S$-graphs} 
	Let $ S $ be a Pouna set. 
	An oriented graph (resp.\ graph) is called an \emph{oriented constrained $S$-graph} (resp.\ a \emph{constrained $S$-graph}) if it is isomorphic to the oriented intersection graph (resp.\ intersection graph) of a non-empty family of $ \mathcal F $ of transformed copies of $ S $ that satisfies Constraints \textup{(C1)-(C5)}, as well as the following constraint: 
	\begin{enumerate}  
		\item[\textbf{\textup{(C6)}}] if $ S $ is strong, then all elements of $ F $ are positive transformed copies of $ S $, and otherwise, they are all positive transformed copies of the horizontal reflection of~$ S $.
	\end{enumerate} 
\end{definition} 

Notice that the set of all oriented constrained $ S $-graphs and the set of all constrained $ S $-graphs both form hereditary classes of graphs. 

\medskip
By definition, every constrained $ S $-graph is a constrained graph. However, as we will see in \Cref{sec:proof-of-equality}, the two classes are indeed equal, and in particular, the class of constrained $ S $-graphs does not change for different sets $ S $.

	Remember that for every Pouna set $ S $, either $ S $ or its horizontal reflection is a strong Pouna set (see \Cref{prop:strong-perkins-S-or-horizontal-reflection}). Therefore, restricting our definition of constrained graphs and constrained $ S $-graphs to strong Pouna sets instead of Pouna sets does not reduce the generality of the definition.

Applied to a specific set $ S$, the definition of constrained $ S $-graphs becomes rather intuitive. For example, when $ S $ is the boundary of a rectangle in $ \RR^2 $, constrained $ S $-graphs are exactly \emph{strict frame graphs}. Also, when $ S $ is a non-vertical and non-horizontal line segment, constrained $ S $-graphs are exactly \emph{strict line-segment graphs} (defined in Section 6 of~\cite{BG1PournajafiTrotignon2021}). 

See \Cref{fig:examples-constrained-S-graphs} for two more examples of constrained $ S $-graphs where $ S $ is a circle and when $ S $ is a square that is not axis-aligned. In each row of the figure, from left to right, the pictures represent the following: 
\begin{itemize} 
	\item The first picture shows the set $ S $ (in solid color) and its territory (hatched). For the rest of the figures, we have not shown the territory anymore. 
	\item The second picture shows the way that two sets can intersect, i.e. what is described by Constraint (C1). 
	
	\item The third picture represents Constraint (C2). In other words, it shows that if two sets do not intersect but one has an intersection with the territory of the other, how they must be placed. Notice that in the first line, there are two possibilities to place a transformation of the circle in the territory of the other transformation of the circle with no intersection.  
	
	\item The fourth picture shows the forbidden construction in Constraint (C3).  
	
	\item The fifth picture shows the forbidden construction in Constraint (C4).  
	
	\item Finally, we must keep in mind that there must not be three distinct sets that mutually intersect.  
\end{itemize}

\begin{figure} 
	\centering 
	\vspace*{-1cm} 
	\includegraphics[width=12cm]{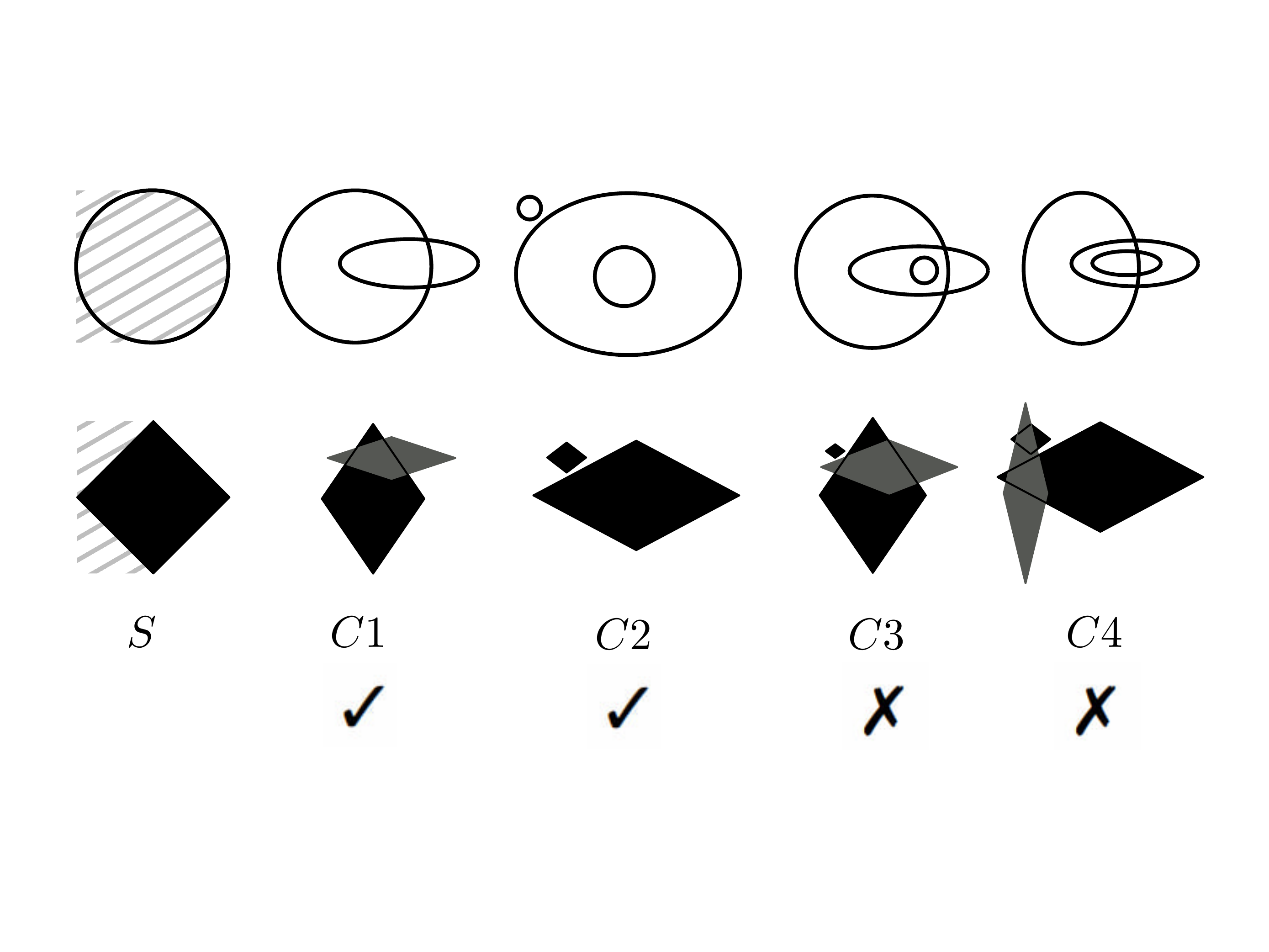} 
	\vspace*{-1.8cm} 
	\caption{Examples of constrained $ S$-graphs} \label{fig:examples-constrained-S-graphs} 
\end{figure}

\Cref{fig:examples-of-constrained-S-graph} shows that $ C_6$ and $ K_{3,3}$ are Constrained $S$-graphs for $ S $ equal to a circle and a positively sloped line-segment respectively. 

\begin{figure}
	\centering
	\vspace*{-1.5cm} 
	\includegraphics[width=11cm]{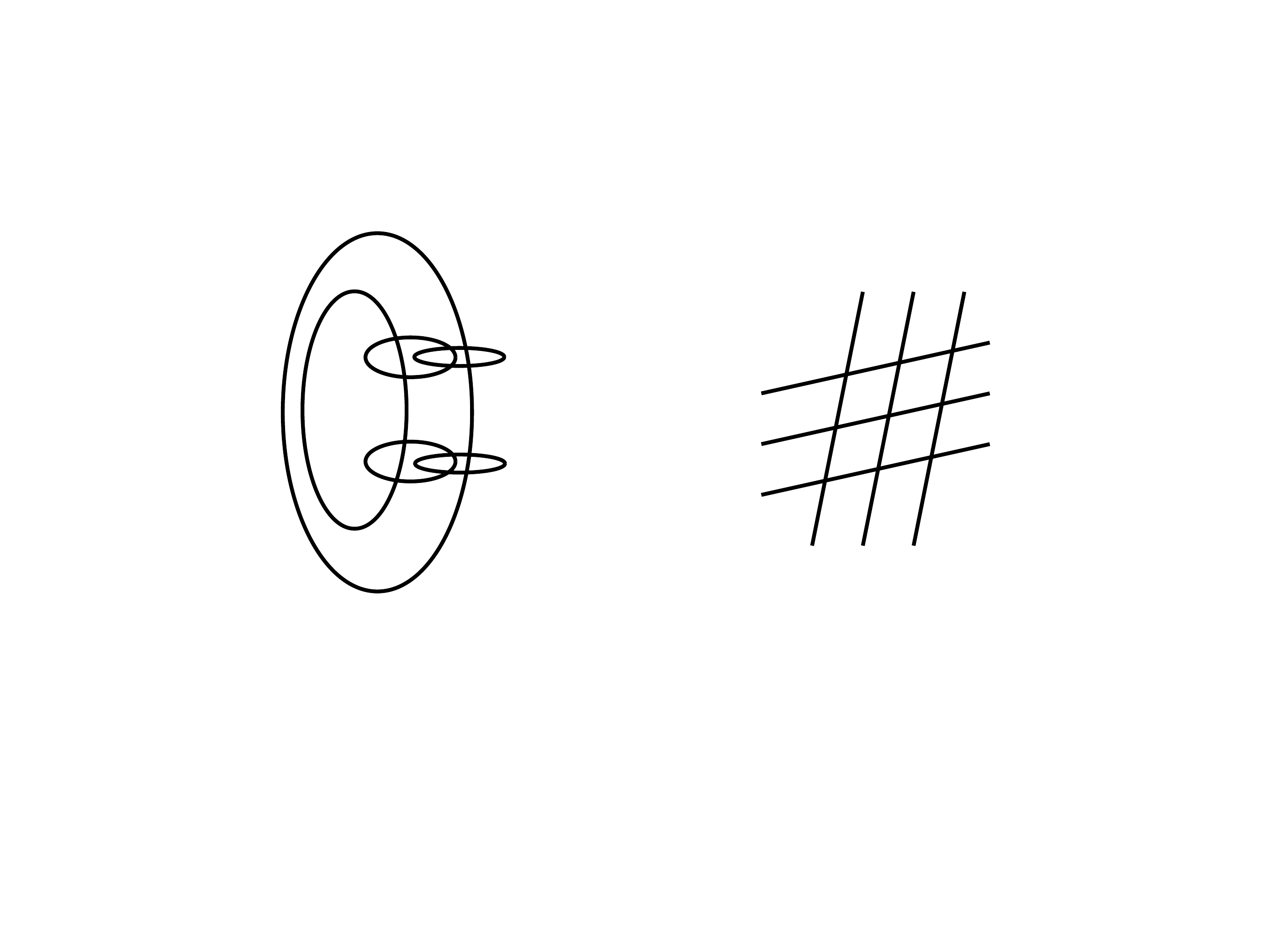}
	\vspace*{-2.8cm} 
	\caption{Left: $C_6 $ as a constrained circle graph. Right: $K_{3.3}$ as a constrained positively sloped line-segment graph.} \label{fig:examples-of-constrained-S-graph}
\end{figure}

\section{Equality of the three classes} \label{sec:proof-of-equality}

Let us now state our main theorem more precisely. 

\begin{theorem} \label{thm:main-complete-theorem}
Let $ G $ be an oriented graph, the followings are equivalent: 
\begin{enumerate}
	\item $ G $ is a constrained graph.
	\item $ G $ is an abstract Burling graphs (and equivalently, a Burling graph). 
	\item $ G $ is a constrained $ S $ graphs for some Pouna set $ S $.
\end{enumerate}
In particular, the following classes of non-oriented graphs are equal: Burling graphs, constrained graphs, and constrained $ S $-graphs for any Pouna set $ S $.
\end{theorem}

We prove $(1) \implies (2) $ in~\Cref{sec:proof-const.-implies-abstractBG}, and $(2) \implies (3)  $ in~\Cref{sec:proof-BG-implies-const.S-graph}. 
Then we conclude the proof. 

\subsection{Constrained graphs are abstract Burling graphs} \label{sec:proof-const.-implies-abstractBG}

In this Section, we first prove some more properties of Pouna sets anf then use these properties to show that constrained graphs are Abstract Burling graphs. 


\begin{lemma} \label{lem:prec-properties} 
	Let $ A $ and $ B $ be two strong Pouna sets. If $ A \prec B $, then  
	\begin{enumerate} 
		\item $\rset A  < \rset B  $, 
		\item $\hset{A} \leq \hset{B} $.
	\end{enumerate}  
\end{lemma} 
\begin{proof} 
	To prove (1), let $ r = \rset A  $. Because $ A$ is compact, there exists a point $ (r, y) $ in $ A $. Since $ A\prec B$, we have $(r,y) \in \Ter B $. Therefore, there exists $ r' $ such that $ r'> r $ and $ (r', y) \in B $. Notice that, $  r' \leq \rset B$. Hence, $ \rset A < \rset B $.

	To prove (2), notice that $ A \subseteq \boxset{A} \subseteq \Ter{B} \subseteq \boxset{B}$. So, $ \bset{A} \geq \bset{\boxset{B}} = \bset{B}$ and $ \tset{A} \leq \tset{\boxset{B}} = \tset{B} $. Therefore, $ \hset{A} = \tset{A} - \bset{A} \leq \tset{B} - \bset{B} = \hset{B}$. 
\end{proof}

We say that two strong Pouna sets $A$ and $ B $ are \emph{comparable}\index{Pouna sets!comparable} if one of the following happens: $ A \adj B $, $ B \adj A $, $ A \prec B $, or $ B \prec A $. 

\begin{lemma} \label{lem:comparable-Pouna-sets} 
	Let $ A $ and $ B $ be two strong Pouna sets in a family $\mathcal F $ which satisfies Constraints \emph{(C1)} and \emph{(C2)}. If $ \Ter A  \cap \Ter B \neq \varnothing $, then $ A$ and $ B $ are comparable.   
\end{lemma} 
\begin{proof} 
	If $ A \cap B \neq \varnothing $, then by Constraint (C1), either $ A \adj B $ or $ B \adj A $. So, we may assume $ A \cap B = \varnothing $. Choose a point $ p = (x,y) \in \Ter A \cap \Ter B  $. There exists $ x', x'' \in \RR $, both bigger than $ x $, such that $ p'=(x', y) \in A $ and $ p''=(x'', y) \in B $. Since $ A $ and $ B $ are disjoint, $ x' \neq x''$. First, assume that $ x''> x' $. Notice that $ p' \notin B $ and that $ p'$ is on the straight line joining $ p $ and $ p'' $, which are both points in $ \boxset B  $. Therefore, $ p' \in \boxset B $. Consequently, $ p' \in \Ter B $. Therefore $ A \cap \Ter B \neq \varnothing $, and by Constraint (C2), we have $ A \prec B $. Second, assume that $ x''<x $. With a similar argument, we deduce $ B \prec A $.  
\end{proof}

Now we can prove that oriented constraint graphs are abstract Burling graphs. 

\begin{theorem} \label{thm:equalthm-constrained-implies-abstract-Burling} 
	Every oriented constrained graph is an oriented abstract Burling graph and therefore a Burling graph.  
\end{theorem} 

\begin{proof} 
	Let $ G $ be an oriented constrained graph. So, $ G $ is the oriented intersection graph of a non-empty and finite family $ \mathcal F $ of strong Pouna sets which satisfies Constraints (C1)-(C5). We prove that $ (\mathcal F, \prec, \adj) $ is a Burling set.
	
	{\noindent \textbf{Calim.} \textit{The relation $\prec$ is a strict partial order.}}
	
	By Lemma~\ref{lem:prec-properties}, if $ A \prec B $, then $ r(A) < r(B) $. This implies that $ \prec $ is antisymmetric. 
	
	Now assume that $ A \prec B $ and $ B \prec C $. So, $ \boxset{A} \subseteq \Ter{B} \subseteq \boxset{B} \subseteq \Ter{C}$. Thus $ A \prec C $. So, $\prec $ is transitive.

	Being antisymmetric and transitive, $\prec$ is a partial order.  
	
	{\noindent \textbf{Calim.} \textit{The relation $\adj$ has no directed cycles.}}
	
	If $  A \adj B $, then by definition, $r(B) < r(A)$. Thus, $ \adj $ cannot have any directed cycles.

	{\noindent \textit{Claim}. \textit{Axiom \emph{(A1)} holds.}}
	
	Let $ A \prec B $ and $ A\prec C $. So, $ A \subseteq \Ter B \cap \Ter C$, and in particular, $ \Ter B \cap \Ter C \neq \varnothing $. So, by Lemma~\ref{lem:comparable-Pouna-sets}, $ B $ and $ C $ are comparable. However, because of Constraint (C3), we have $ B \cap C = \varnothing $. So, either $ B \prec C $ or $ B\prec C $.  
	
	{\noindent \textit{Claim}. \textit{Axiom \emph{(A2)} holds.}}
	
	Let $ A \adj B $ and $ A \adj C $. So, the set $ \{(x,y) \in A: x = \lset A  \} $ is a subset of both $ \Ter B  $ and $ \Ter C $. In particular, $ \Ter B \cap \Ter C \neq \varnothing $, and therefore by Lemma~\ref{lem:comparable-Pouna-sets}, $ B $ and $ C $ are comparable. However, because of Constraint (C5), we have $ B \cap C = \varnothing $.  Therefore, either $ B \prec C $ or $ B\prec C $.

	{\noindent \textit{Claim}. \textit{Axiom \emph{(A3)} holds.}}
	
	Let $ A \adj B $ and $A \prec C $. Hence, by definition, $ \rset B \leq \rset A $, and by Lemma~\ref{lem:prec-properties}, $ \rset A < \rset C $. Consequently, $ \rset B < \rset C$. So, if $ B \cap C \neq 0 $, we must have $ C \adj B $. But then $ A \adj B$, $ C \adj B $, and $ A \prec C $ contradict Constraint (C4). Thus, $B \cap C  = \varnothing $. Now, choose a point $ p $ in $ A \cap B $. Since $ A\subseteq \Ter C $, we have $ p \in \Ter C $. Hence,  $ B  \cap \Ter C \neq \varnothing $. Therefore, by Constraint (C2), we have $ B \prec C $. 
	
	{\noindent \textit{Claim}. \textit{Axiom \emph{(A4)} holds.}}
	
	Let $ A \adj B $ and $ B \prec C $.  So, by definition of $\adj$, we have $ \hset A < \hset B $, and by Lemma~\ref{lem:prec-properties}, we have $ \hset B < \hset C$. So, $ \hset A < \hset C$. Hence, if $ A \cap C  \neq \varnothing $, we have $ A \adj C $. On the other hand, if $  A\cap C = \varnothing$, since $ A \cap B \neq \varnothing $ and $ B \subseteq \Ter C $, we have $ A \cap \Ter C \neq \varnothing $. Therefore, by constraint (C2), $ A \prec C $. 
	
	So, $(\mathcal F, \prec, \adj) $ is a Burling set. 
	Finally, because of Constraint (C1), the oriented abstract Burling graph $ \hat G $ obtained from the Burling set $ (\mathcal F, \prec, \adj) $ is indeed isomorphic to $ G $, the oriented intersection graph of $ \mathcal F $. So, $ G $ is an oriented abstract Burling graph.  
	
	Finally, by \Cref{thm:PTabstract-BG}, the graph $ G $ is also an abstract Burling graph. 
\end{proof}

\subsection{Burling graphs are constrained $ S$-graphs} \label{sec:proof-BG-implies-const.S-graph}

In \cite{pawlik2013general}, Pawlik, Kozik, Krawczyk, Laso\'n, Micek, Trotter, and Walczak introduce, for every Pouna set $ S $, a sequence $ \{\mathcal F_k\}_{k \geq 1}$ of families of transformed copies of $ S $ such that the class generated by the (non-oriented) intersection graphs of $\mathcal F_k$'s is equal to the class of Burling graphs, thus showing that Burling graphs are $ S $-graphs. 

In this section, we repeat the construction from~\cite{pawlik2013general} (in an oriented version and with slightly different terminology and details so it matches our earlier definitions), we show that it satisfies all Constraints (C1)-(C6), which implies that Burling graphs are constrained $ S $-graphs.

We first, however, need some lemmas and some more properties of Pouna sets. 


\medskip

The proofs of Lemmas \ref{lem:crossing-between-two-lines}, \ref{prop:crossing-a-subrectangle}, and \ref{prop:horizont-and-vertical-crossing-intersect} can be found in \Cref{appendix-proofs}.

Let $ R $ be an axis-parallel rectangle.  Let $ A \subseteq \RR^2 $. We say that $ A $ \emph{crosses $ R $}\index{crossing} vertically (resp.\ horizontally) if there exists a  $ \gamma: [0,1] \rightarrow A \cap R $ such that $ \gamma(0) $ and $\gamma(1)$ are respectively on the bottom-side and on the top-side (resp.\ on the left-side and on the right-side) of $ R$.

\begin{lemma} \label{lem:crossing-between-two-lines} 
	Let $ y_0, y_1 \in \RR$ such that $ y_0 \leq y_1 $. For $ i  \in \{0,1\}$, let $ L_i $ denote the line $ y = y_i $ in $\RR^2$. Let $ \gamma:[0,1] \rightarrow \RR^2 $ be a continuous function such that for $ i \in \{0,1\} $, we have $ \gamma(i) \in L_i$. Then, there exist  $x_0, x_1 \in [0,1] $ such that $ x_0 \leq x_1 $ and the following hold: 
	\begin{itemize}
		\item the path $ \gamma' = \gamma|_{[x_0, x_1]} $ is  always between or on the lines $ L_0 $ and $L_1$, i.e.\ $\im{\gamma'} \subseteq \{ (x,y) : y_0 \leq y \leq y_1 \} $,  
		\item for $ i \in \{0,1\} $, we have $ \gamma'(i) \in L_i$.
	\end{itemize}
\end{lemma} 

\begin{lemma} \label{prop:crossing-a-subrectangle} 
	Let $ R $ and $ R' $ be two axis-aligned rectangles such that:  
	\begin{itemize} 
		\item $ \lset{R'} \leq \lset{R} \leq \rset{R} \leq \rset{R'}$,  
		\item $\bset{R} \leq \bset{R'} \leq \tset{R'} \leq \tset{R} $. 
	\end{itemize} 
	If a set $ A $ crosses $ R $ vertically, then it crosses $ R'$ vertically as well.  
\end{lemma} 

\begin{lemma} \label{prop:horizont-and-vertical-crossing-intersect} 
	Let $R$ be a rectangle in $\RR^2 $. Let $ \alpha: [0,1] \to R$ and $ \beta:[0,1]\to R $ be two paths joining the bottom side of $ R $ to its top side and the left side of $ R $ to its right side respectively. Then $\im{\alpha} \cap \im{\beta} \neq \varnothing $.  
\end{lemma}


The following lemma shows that strong Pouna sets and their territories behave well under positive transformations. For some properties that are easy to check, we have provided the proof in \Cref{appendix-proofs}.

\begin{property} \label{prop:ter(T)=T(ter)} Let $ S $ be a strong Pouna set and $ T $ be a positive transformation. Then, $ \Ter{T(S)} = T(\Ter{S})$. In particular, $ T(S) $ is strong.  
\end{property} 

The proof is in \Cref{appendix-proofs}. 


Let $ B $ and $ E $ be two rectangles such that $ E \subseteq R $. The right-extension\index{right extension} of $ E $ in $ R $ is the rectangle $ E_r $ defined as follows: $$ E_r = [\rset{E}, \rset{R}] \times [\bset{E}, \tset{E}]. $$ See \Cref{fig:right-extension}.

\begin{figure} 
	\centering 
	\vspace*{-1cm} 
	\includegraphics[width=6cm]{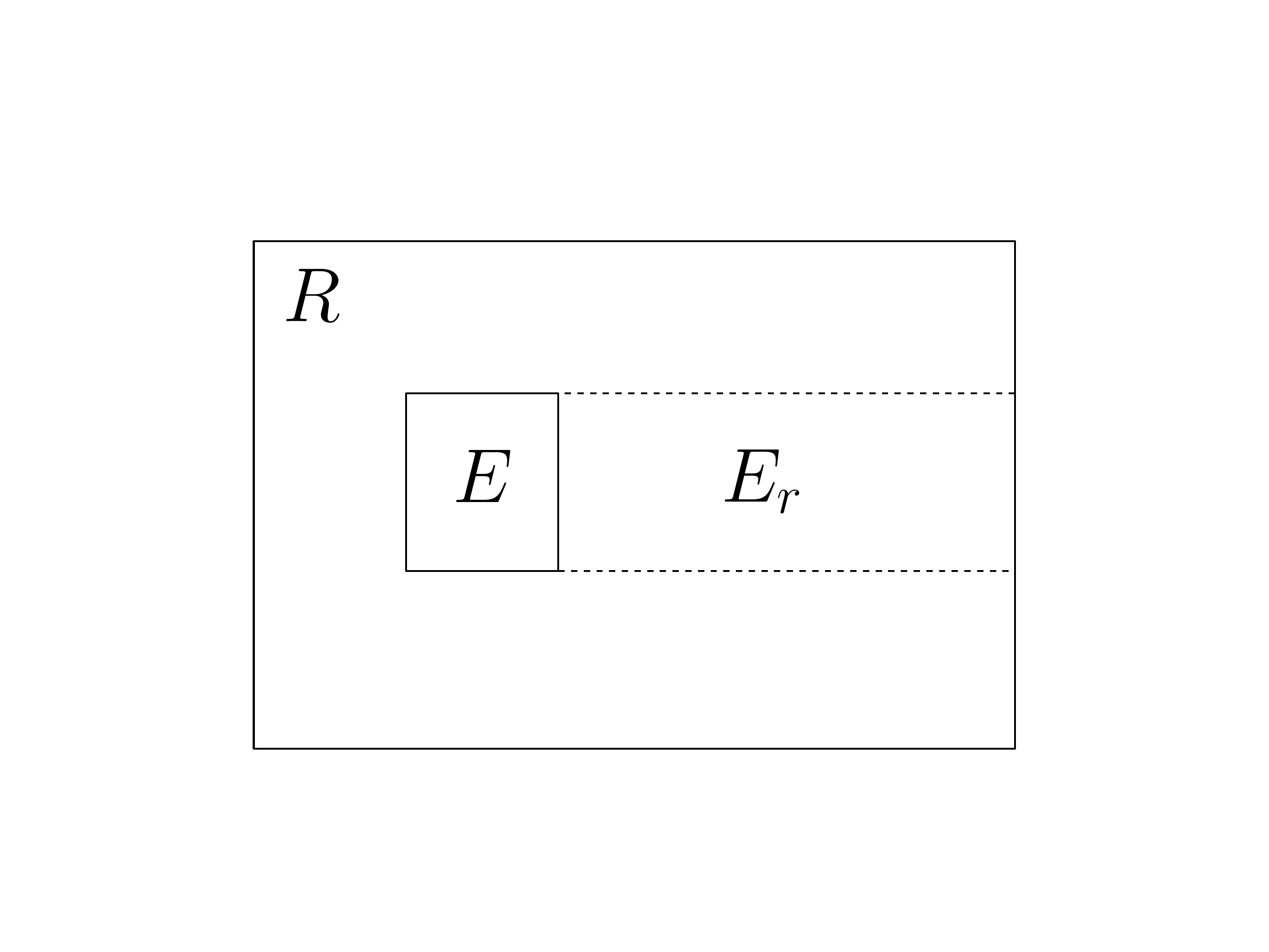} 
	\vspace*{-1cm} \caption{$E_r $ is the right extension of $ E $ in $ R $.} 
	\label{fig:right-extension} 
\end{figure}

\begin{definition}
	A \emph{subterritory}\index{subterritory} for a strong Pouna set $ S $ is a non-empty closed rectangle $ E $ such that 
	\begin{enumerate} 
		\item $ E \subseteq \Ter{S} $, 
		\item $ \lset{E} > \lset{S} $, $ \rset{E} < \rset{S}$, $\bset{E} > \bset{S} $, and $ \tset{E} < \tset{S}$, 
		\item $ S $ crosses the right extension of $ E $ vertically.  
	\end{enumerate}
\end{definition}

A strong Pouna set always has a subterritory, as we prove in the following lemma. 

\begin{lemma} \label{lem:subterritory-exists} 
	Every strong Pouna set has a subterritory. 
\end{lemma} 

\begin{proof} 
	Let $ S $ be a strong Pouna set and let $ B = \boxset{S} $. 
	By \Cref{prop:int-box-min-S-non-empty} there exist a point $ p = (x_p,y_p) \in B^\circ \sm S $. So, there is $ \epsilon > 0 $ such that $ D(p,\epsilon) \subseteq B^\circ \sm S $.  
	
	Let $L_P $ be the ray $\{ (x,y) : y = y_p, x \geq x_p \}$. Notice that $ L_P \cap S $ is non-empty and compact. Let $ s = (x_s, y_s)$ be the point in $ L_P \cap S $ which obtains the value $ \lset{L_P \cap S} $. Notice that $ y_s = y_p$. Consider the following rectangle in $ B $: $$ R = [x_S - \epsilon/2, \rset{S}] \times [y_s - \epsilon , y_s + \epsilon]. $$
	See \Cref{fig:ptoof-subter-exists}. 
	
	\begin{figure} 
		\centering 
		\vspace*{-1cm} 
		\includegraphics[width=9cm]{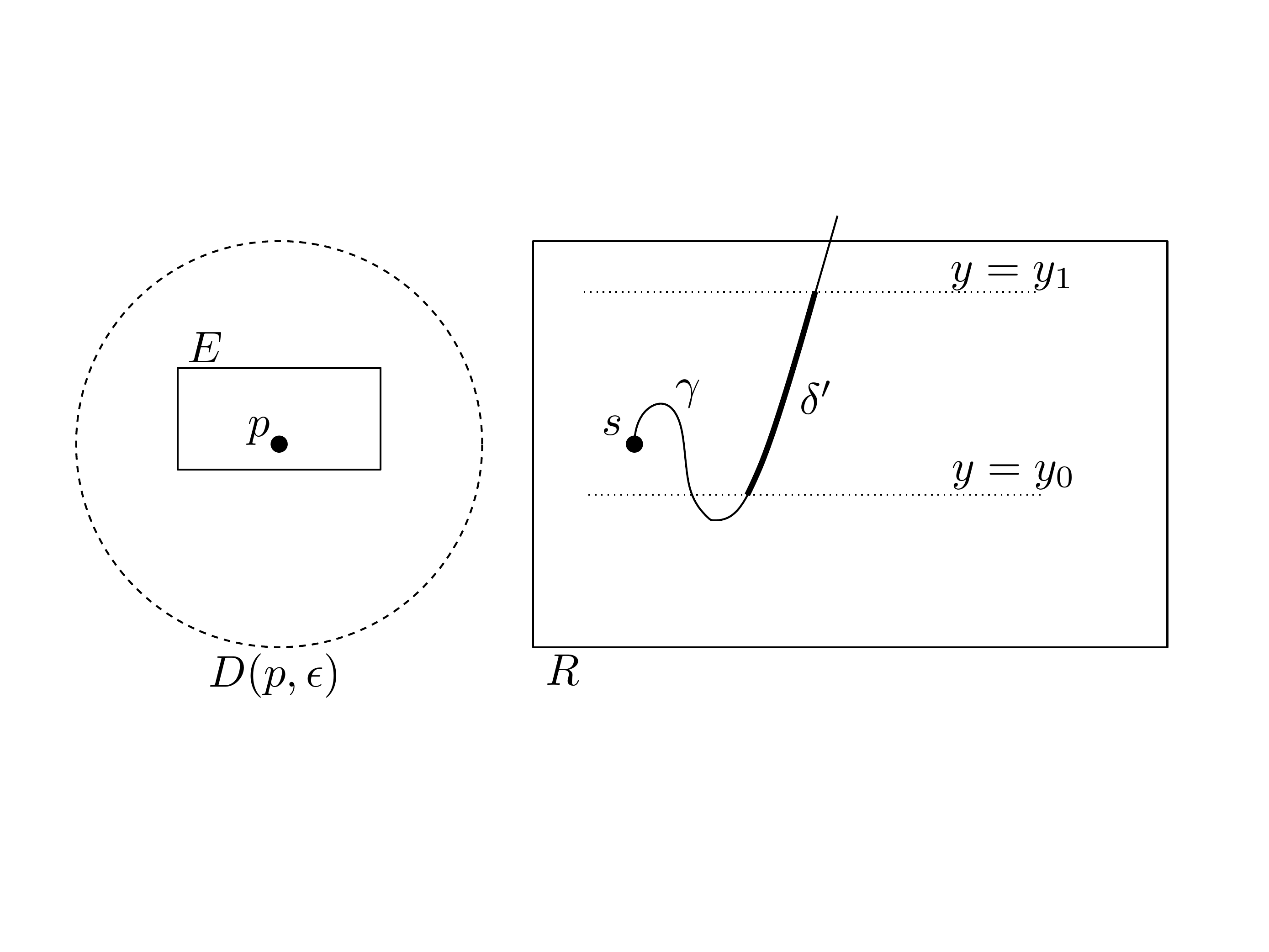} 
		\vspace*{-2cm} 
		\caption{For proof of Lemma~\ref{lem:subterritory-exists}.} \label{fig:ptoof-subter-exists} 
	\end{figure}
	
	In particular, $ s \in R^\circ $. Also, $ R = \bar{R} $ does not intersect the border of $ B $. On the other hand, there is a point $ s' $ of $ S $ on the top-side of $ B $. Since $ S $ is a path-connected set, we must have a path $\gamma$ from $ s $ to $ s' $. By \Cref{lem:real-border-lemma}, the image of $\gamma$ must intersect $\partial B $, and in particular in a point other than $(x_S - \epsilon/2, y_s)$ and $ (\rset{S}, y_s)$. So, $ \im{\gamma} \cap B $ is not a horizontal line. In particular, there are $ y_0, y_1 \in \RR $ such that $ y_s - \epsilon \leq y_0 < y_2 \leq y_s + \epsilon$ and such that there is a path $\delta$ in $ R $ joining a point on the line $ y = y_0 $ to a point on the line $ y = y_1 $.
	
	So, by \Cref{lem:crossing-between-two-lines}, applied to  $\delta $, there is a path $\delta': [0,1] \rightarrow \RR$ such that $ \pi_2(\delta(0)) = y_0 $, $ \pi_2(\delta(1)) = y_1 $, and $ \im{\delta'} \subseteq [x_S - \epsilon/2, \rset{S}] \times [y_0 , y_1]$.
	
	Now, let $ E $ be a rectangle entirely inside $ D(p,\epsilon)$ defined as follows: $$ E = [x_p - \epsilon/2, x_p + \epsilon/2] \times [(y_p+y_0)/2, (y_p+y_1)/2]. $$
	
	Notice that by \Cref{prop:crossing-a-subrectangle}, $ \delta' $ crosses the right extension of $ E $ vertically. Clearly, $ E$ satisfies all other properties of subterritory as well. So, $ E $ is a subterritory of $ S $. 
\end{proof}

The next property, whose proof is in \Cref{appendix-proofs}, shows that subterritories behave well under positive transformations.

\begin{property} \label{prop:T(E)-subter-of-T(S)} If $ E $ is a subterritory of a strong Pouna set $ S $, then for every positive transformation $ T $, we have that $ T(E) $ is a subterritory of $ T(S) $.  
\end{property}

Finally, the following property states that positive transformations preserve the Constraints \textup{(C1)-(C6)} of the definition of constrained $ S$-graphs.

\begin{property} \label{prop:T(F)-satisfies-C1-C5} Let $ S $ be a strong Pouna set, and let $ \mathcal F $ be a finite family of transformed copies of $ S $ satisfying Constraints \emph{(C1)-(C6)}, then for every positive transformation $ T $ the family $ \{T(S): S \in \mathcal F\} $ also satisfies \emph{(C1)-(C6)}. \end{property}

Again, the proof is in \Cref{appendix-proofs}.

\subsection*{Construction of Pawlik, Kozik, Krawczyk, Laso\'n, Micek, Trotter, and Walczak} 

Let us restate the construction of \cite{pawlik2013general} in a slightly different manner.

Let $ S $ be a Pouna set, and let $ \mathcal F $ be a finite family of transformed copies of $ S $. Set $ \mathfrak{B} = \boxset {\mathcal F} $.

A \emph{prob}\index{prob} for $ \mathcal F $ is a closed rectangle $ P$ such that: $P \subseteq \mathfrak{B} $ and $ \rset{P} = \rset{\mathfrak{B}}$.  We denote the set $ \{A \in \mathcal F: A \cap P \neq \varnothing \} $ by $ N_{\mathcal F}(P) $, or $ N(P) $ if there is no confusion. 

Let $ P $ be a prob for $ \mathcal F $. A \emph{root}\index{prob!root of} of $P $ is a rectangle of the form $ \{(x,y) \in P: x \leq x_0 \} $, for some $ x_0 \in (\lset P, \rset P) $, which does not intersect any element of  $\mathcal F $. Notice that not every prob has a root, and that when a prob has a root, it has infinitely many roots. Moreover, the roots of a prob form a totally ordered set with inclusion. 

The prob $ P$ is said to be \emph{stable}\index{prob!stable} if: 
\begin{enumerate} 
	\item $P$ has a root, and there exists a root $ R $ of $ P$ such that for every $ A \in N(P)$, we have $ R \subseteq \Ter{A}$, 
	\item the elements of $ N(P) $ are mutually disjoint, 
	\item for every $ A \in N(P)$, we have $ \bset{A} < \bset{P} $ and $ \tset{P} < \tset{A}$, 
	\item every $ A \in N(P)$ crosses $ P$. 
\end{enumerate}

\begin{remark} It is worth mentioning that the fourth item in the definition of stable prob does not follow from the three other items. In \Cref{fig:4th-item-stable-prob}, a prob $P$ with a root $ R $ and $N(P) = \{A\} $ are shown. All items 1-3 of the definition hold here, but not item 4.  
\end{remark}

\begin{figure} 
	\centering \vspace*{-1.7cm} 
	\includegraphics[width=8.2cm]{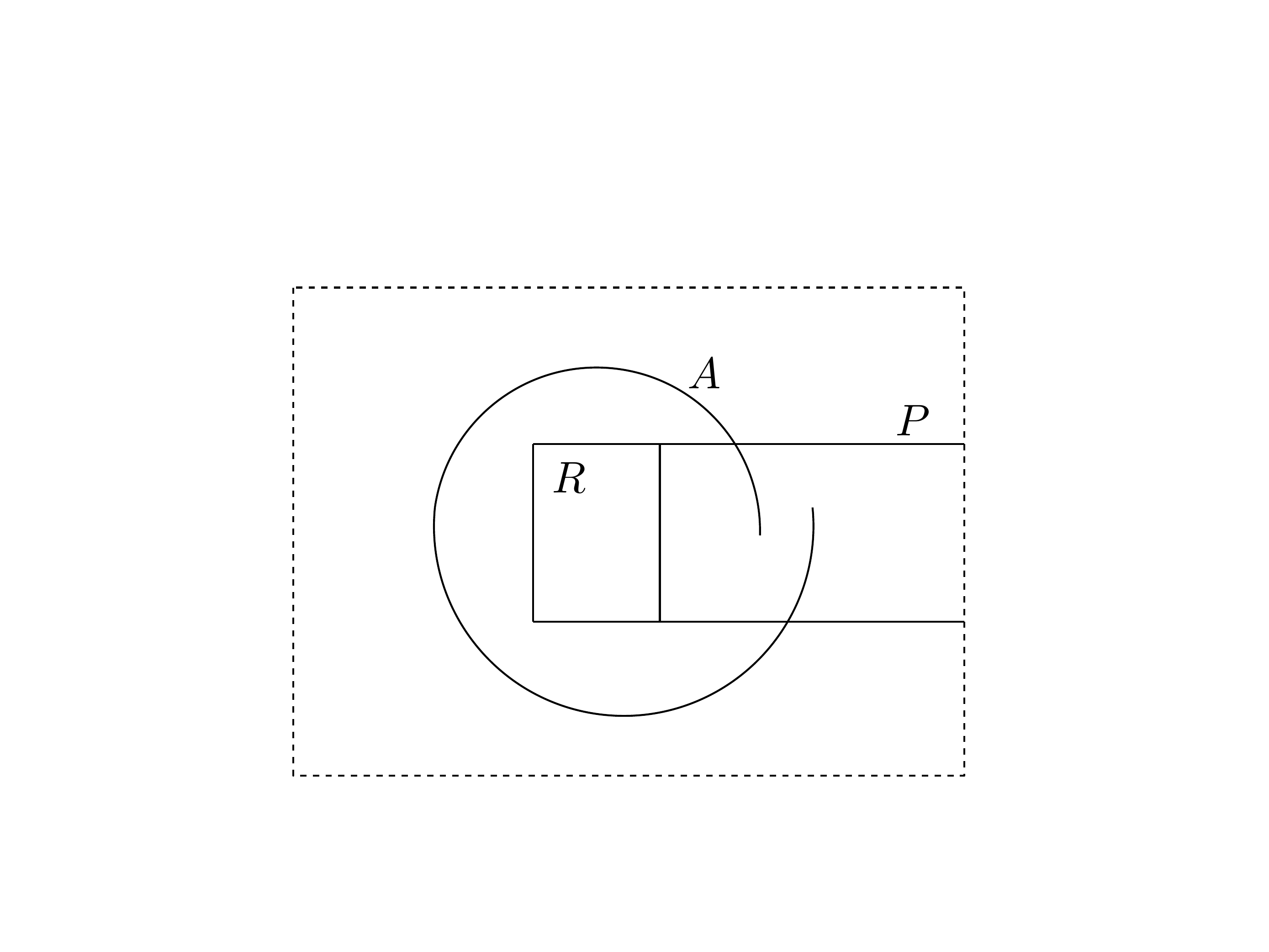} 
	\vspace*{-1.3cm} 
	\caption{The fourth item in the definition of stable prob does not hold here.} \label{fig:4th-item-stable-prob} 
\end{figure}

What we call a stable prob in this article is similar to what is called a \emph{prob} in~\cite{pawlik2013general}.

\begin{property} \label{prop:every-root-stable} 
	Let $ P$ be a stable prob for a family $\mathcal F $ of Pouna sets. Then, for every root $ R $ of $P $ and for every $ A \in N(P)$, we have $ R \subseteq \Ter A $.  
\end{property} 
\begin{proof} Let $ R_0 $ be the root in the definition of stable prob. If $ R \subseteq R_0 $, the result is obvious. If not, let $ p = (x,y) \in R \sm R_0 $. So, there exist $ x_0 < x$ such that $ p_0 = (x_0, y) \in R_0 $ So, in particular $p_0 \in \Ter{A} $. Also, $ p_0 \in R $, because $ R_0 \subseteq R $. So, there exists $ x' > x_0 $ such that $ p' = (x', y) \in A $. Since $ p' \in A $, we have $ p' \notin R $. So, in particular, $ x' \neq x $.  If $ x' < x $, then $ x' $ is on the strait line joining $ p_0 $ and $ p $. But $ p_0, p \in R $ and $ R $ is convex, so $ (x',y) \in R $, a contradiction. Hence $ x'> x $. Now, to show that $ p\in \Ter{A}$, it is enough to show that $ p \in \boxset{A} \sm A $. But $ p $ being in  $R $, is not in $ A $. On the other hand, $ p $ is in on the straight line between $ p_0 $ and $ p' $. Now because $ p_0 \in \Ter{A} \subseteq \boxset{A} $ and $ p' \in A \subseteq \boxset{A} $, we have $p \in \boxset{ A }$. This completes the proof.   
\end{proof}

Let $ E $ be a rectangle in $ \boxset{\mathcal F} $. The prob \emph{defined by $ E $}\index{prob!defined by a rectangle} in $ B $ is the prob $ P$ which is obtained by extending the right side of $ E $ to reach the border of $B $, i.e. $ P = \{(x,y) \in B: \lset E \leq x \leq \rset B, \bset E \leq y \leq \tset E \}. $ Notice that if $  E $ does not intersect any member of $ \mathcal F $, then it is a root for $ P $.

From now on, fix a strong Pouna set $ S $ and a subterritory $ E $ of $ S $. Also, from now on, for the transformed copy $ S'=T(S) $, we consider the subterritory $ T(E) $. Refer to \Cref{fig:PKKconstruction} for a visualization of the construction described hereunder, applied to a very simple 1-element family of a set $ S $.

\begin{figure}
	\centering
	\includegraphics[width=7cm]{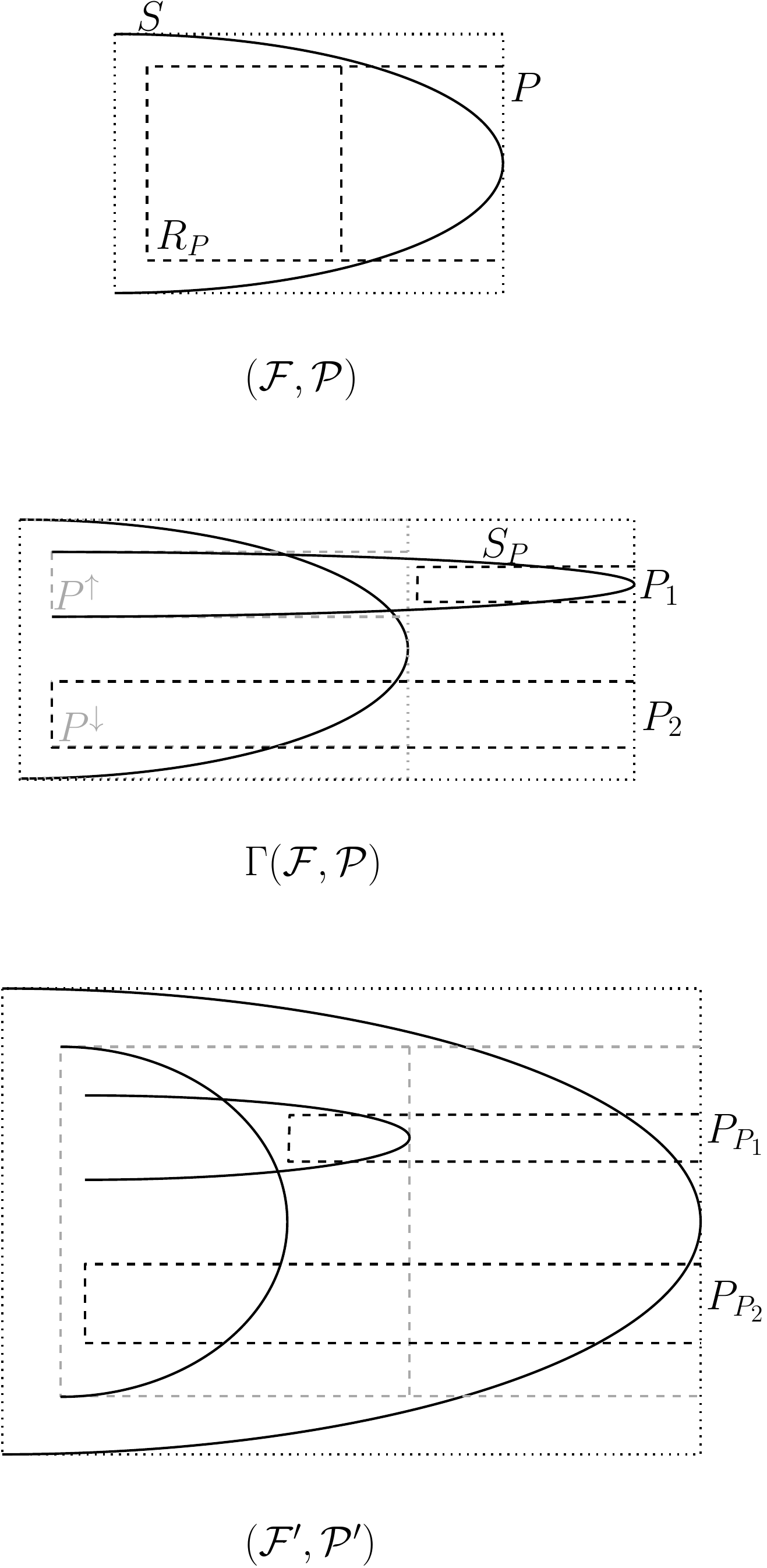}
	\caption{Construction of~\cite{pawlik2013general} applied to $(\mathcal F, \mathcal P)$. The second line presents $\Gamma(\mathcal F, \mathcal P)$ and the third line presents $(\mathcal F', \mathcal P') = \nextF(\mathcal F, \mathcal P)$. The shapes in gray are not parts of the object and are presented to make following the construction easier. The scales have changed to make the illustration clearer.} \label{fig:PKKconstruction}
\end{figure}

	Let $(\mathcal F, \mathcal P) $ be a tuple where $ \mathcal F$ is a family of transformed copies of $ S $ and $ \mathcal P $ is a set of probs of $\mathcal F$. We define an operation $\Gamma$ where $ (\mathcal F', \mathcal P') = \Gamma(\mathcal F, \mathcal P) $ is obtained as follows: 
	\begin{enumerate} 
		
		\item[\textbf{\textup{(S$'$1)}}] For every $ P \in \mathcal P $, let $P^\uparrow $ and $ P^{\downarrow} $ be respectively the top one-third and the bottom one-third of $ P $, i.e. $$ P^\uparrow = [\lset P, \rset P] \times [\frac{\bset P + 2\tset P}{3}, \tset P] $$ and $$ P^\downarrow = [\lset P, \rset P] \times [\bset P, \frac{2\bset P+ \tset P}{3}]. $$ 
		
		\item[\textbf{\textup{(S$'$2)}}] Set $ S_P $ to be a transformed copy of $ S $ where we first match the boundary of $ \boxset S $ on the boundary of $ P^\uparrow $, and then we scale it horizontally by $ \frac{2 \wset{S}}{\lset E - \lset S} $ keeping the left-side of $\boxset S $ fixed. Formally, the transformation described above is $ T_P = T_2 \circ T_1 : \RR^2 \rightarrow \RR^2 $, where $$ T_1(x,y) =  \Big(  \frac{\wset{P^\uparrow}}{\wset{S}}x + \lset{P^\uparrow} - \frac{\lset{S}\wset{P^\uparrow}}{\wset{S}} ,  \frac{\hset{P^\uparrow}}{\hset{S}}y + \bset{P^\uparrow} - \frac{\bset{S}\hset{P^\uparrow}}{\hset{S}} \Big) $$ and  $$ T_2(x,y) =  \Big( \frac{2 \wset S}{\lset E - \lset S}x + \lset{P^\uparrow}(1 - \frac{2 \wset{S}}{\lset E - \lset S}) , y \Big). $$ This transformation ensures that the subterritory of $ S_P$, i.e.\ $ T_P(E) $, is outside~$ \boxset{\mathcal F} $ (See \Cref{prop:prop-after-def-Gamma}). Denote $ T_P(E) $ by $ E_P$. 
		
		\item[\textbf{\textup{(S$'$3)}}] Set $ \mathcal F' = \mathcal F \cup \big( \cup_{P \in \mathcal P} S_P  \big) $.  
		
		\item[\textbf{\textup{(S$'$4)}}] For $ P \in \mathcal P $, denote by $ P_1 $ the prob for $\mathcal{F'}$ defined by $ E_P $, and denote by $P_2 $ the prob for $ \mathcal{F'} $ defined by $ P^\downarrow $. 
		
		\item[\textbf{\textup{(S$'$5)}}] Set $ \mathcal P' = \{ P_1, P_2 : P \in \mathcal P \}$. 
	\end{enumerate} 

\begin{definition} \label{def:const-pawlik-to-walzack}
	Let $ S $ be a strong Pouna set. Let $\mathcal F $ be a family of positive transformed copies of $ S $, and let $ \mathcal P $ be a set of its probs. We define $\nextF(\mathcal F, \mathcal P )$ as follows.
	\begin{enumerate} 
		\item[\textbf{\textup{(S1)}}] Set $ (\mathcal F_0, \mathcal P_0) = \Gamma(\mathcal F, \mathcal P)$. 
		
		\item[\textbf{\textup{(S2)}}] For every $ P \in \mathcal P $, choose a root $ R_P$. (To see that $P$ has a root, see~\cite{pawlik2013general} or \Cref{thm:equalthm-Burling-is-constraintS}.) Create a transformed copy $ (\mathcal F^P, \mathcal P^P)$ of $(\mathcal F_0, \mathcal P_0)$ such that $ \boxset{\mathcal F^P} $ is matched to $ R_P $. Formally, apply the transformation: $$ T'_P(x,y) = \Big( \frac{\wset{R_P}}{\wset{B_P}}x + \lset{R_P} - \frac{\lset{B_P}\wset{R_P}}{\wset{B_P}} , \frac{\hset{R_P}}{\hset{B_P}}y + \bset{R_P} - \frac{\bset{B_P}\hset{R_P}}{\hset{B_P}} \Big), $$ where $ B_P = \boxset{\mathcal F^P} $. 
		
		\item[\textbf{\textup{(S3)}}] Set $ \mathcal F' = \mathcal F \cup \big(\cup_{P \in \mathcal P}  \mathcal F^P \big) $. 
		
		\item[\textbf{\textup{(S4)}}] Now, for $ P \in \mathcal P $ and for $ Q \in \mathcal P^P $, let $ P_Q $ be the prob for $ \mathcal F$ defined by $ Q $.  
		
		\item[\textbf{\textup{(S5)}}] Set $ \mathcal P' = \{P_Q: P \in \mathcal P, Q \in \mathcal P^P \}. $
	\end{enumerate}
	
	The tuple $(\mathcal F', \mathcal P')$ is $ \nextF(\mathcal F, \mathcal P)$. 
\end{definition}

Now, we can define a sequence $\{(\mathcal F_k, \mathcal P_k)\}_{k \geq 1} $ from~\cite{pawlik2013general}, where $ \mathcal F_k $ is a family of positive transformed copies of $ S $, and $ \mathcal P_k $ is a set of probs for $ \mathcal F_k$. 

\begin{definition}[The construction from~\cite{pawlik2013general}] \label{def:the-definition-of-geometric-sequence}
	For $ k =1 $, set $ \mathcal F_1 = \{S\} $ and $ \mathcal P_1 = \{P\} $ where $ P$ is the prob defined by $ E $. For every $ k \geq 1 $, define $(\mathcal F_{k+1}, \mathcal P_{k+1}) = \nextF(\mathcal F_k,\mathcal P_k)$.
\end{definition}

Let us state some properties about the construction of Pawlik, Kozik, Krawczyk, Laso\'n, Micek, Trotter, and Walczak. The proof of Properties \ref{prop:prop-after-def-Gamma} and \ref{prop:more-on-Gamma-and-disjoint-probs} can be found in \Cref{appendix-proofs}.

\begin{property} \label{prop:prop-after-def-Gamma} Adopting the notation from the definition of $ \Gamma $, for every $ P \in \mathcal P $, we have:  
	\begin{enumerate} 
		\item the transformation $T_P $ is positive.   
		\item $ \lset{E_P} > \rset{\boxset{F}} $, so in particular, $ E_P \cap \boxset{F} = \varnothing $. 
	\end{enumerate} 
\end{property}

\begin{property} \label{prop:more-on-Gamma-and-disjoint-probs} 
	Let $ \mathcal F $ be a family of strong Pouna sets, and let $ \mathcal P $ be a set of probs for $ \mathcal F $ that are mutually disjoint. Setting $(\mathcal F',\mathcal P') = \Gamma(\mathcal F, \mathcal P)$ and adopting the notation from the definition of $\Gamma$, we have that for every $ P \in \mathcal P $: 
	\begin{enumerate} 
		\item if $Q \in \mathcal P \sm \{P\}$, then $S_P \cap Q = \varnothing $, $ S_P \cap S_Q = \varnothing $, and $ \Ter{S_P} \cap Q = \varnothing $, 
		\item $N_{\mathcal F'}(P_1) = \{S_P\} $,
		\item $N_{\mathcal F'}(P_2) \subseteq N_{\mathcal F}(P)$ and $ N_{\mathcal F'}(P_2) \subseteq \mathcal F$,
		\item for every $ A  \in \mathcal F' $, we have $ S_P \adj A $ if and only if $ A \in N(P)$, and there exists no $ B \in \mathcal F' $ such that $ B \adj S_P $. 
	\end{enumerate} 
\end{property}

Let $ \mathcal C $ be the class of oriented graphs generated by the intersection graphs of the families $ \mathcal F_{k}$ defined in \Cref{def:the-definition-of-geometric-sequence}. This class is exactly the class of Burling graphs. Indeed, as has been mentioned in \cite{pawlik2013general} (for the non-oriented case) the sequence of the intersection graphs of $ \mathcal F_k$'s is exactly the sequence defined by Burling \cite{Burling1965} in 1965. The proof of this fact is not complicated but is long. As a matter of fact, it is usually stated without a proof in the literature. A sketch of the proof can be found in the Ph.D.\ thesis of the author (see Lemma 5.25 of \cite{PournajafiPhDThesis}). Here, we assume this fact. Therefore, we use the term ``(oriented) Burling graph" for the graph in the class $ \mathcal C $ as well. 

With this assumption, we can prove that for any Pouna set $ S$, an oriented Burling graphs (thus, here, a subgraphs of the oriented intersection graphs of some $ \mathcal F_k$) is a constrained $ S $-graph. First, we need the following lemma.

\begin{lemma} \label{lem:Gamma-preserves-C1-C6}
	Let $ S $ be a strong Pouna set. Let $ \mathcal F $ be a family of transformed copies of $ S $ that satisfies Constraints \textup{(C1)-(C6)}. Let $ \mathcal P$ be a set of mutually disjoint stable probs of $\mathcal{F}$. If $ (\mathcal F', \mathcal P') = \Gamma(\mathcal F, \mathcal P)$, then  
	\begin{enumerate} 
		\item elements of $P'$ are mutually disjoint, 
		\item every element of $P$ is a stable prob for $ \mathcal F'$, 
		\item $ \mathcal F' $ satisfies Constraints \textup{(C1)-(C6)}. 
	\end{enumerate} 
\end{lemma}

\begin{proof} We adopt the notation from the definition of $\Gamma$.
	
	Set $ \mathfrak B = \boxset{\mathcal F}$ and $ \mathfrak{B'} = \boxset{\mathcal F'}$. Notice that $ \lset{\mathfrak{B'}} = \lset{\mathfrak{B}}$, $ \bset{\mathfrak{B'}} = \bset{\mathfrak{B}}$, and $ \tset{\mathfrak{B'}} = \tset{\mathfrak{B}} $. However, $ \rset{\mathfrak{B'}} >  \rset{\mathfrak{B}}$.

	{\noindent \textbf{Claim}. \textit{Elements of $\mathcal P' $ are mutually disjoint.}} 	
	
	Corresponding to every $ P \in \mathcal P $, there are two probs in $ \mathcal P' $, that is, $P_1 $ and $ P_2 $. Notice that $ P_1 \cap P_2 = \varnothing $. So, the fact that the probs in $ \mathcal P' $ are mutually disjoint is implies directly by the same property for $ \mathcal P$.

	{\noindent \textbf{Claim}. \textit{Elements of  $\mathcal P' $ are stable probs for $\mathcal F' $.}}
	
	Fix $ P \in \mathcal P $. We prove that both $ P_1 $ and $P_2 $ are stable probs, and thus every prob in $ \mathcal P' $ is stable. 
	
	First, the prob $ P_1 $ is defined by a subterritory $ E_P $. By Property~\ref{prop:prop-after-def-Gamma}, $E_P \cap \mathfrak{B} = \varnothing $. Therefore, for every $ A \in \mathcal F$, we have $ E_P \cap A = \varnothing $. Moreover, by definition of subterritory, $E_P \cap S_P = \varnothing $. Finally, since $ E_P \subseteq P $, by Property~\ref{prop:more-on-Gamma-and-disjoint-probs}, we have $ E_P \cap S_Q = \varnothing $ for every $ Q \in \mathcal P \sm \{P\} $ as well. Thus $ E_P $ does not intersect any element of $ \mathcal F' $. So, $ E_P $ is a root for $P_1 $. 
	
	Now, notice that by Property~\ref{prop:more-on-Gamma-and-disjoint-probs}, we have $N(P_1) = \{S_P\} $, so item (2) of the definition of stable prob holds. Moreover, since $ E_P $ is a subterritory of $ S_P$, we have  \begin{itemize} \item $ E_P \subseteq \Ter{S_P}$,  \item $ \bset{E_P} > \bset{S_P}$ and $ \tset{E} < \tset{S_P}$, \item $S_P $ crosses $ P_1$ vertically, \end{itemize} which proves item (1), (3), and (4) of the definition of stable prob, respectively. For item (3), we have used the facts that $ \bset{P_1} = \bset{E_P} $ and $ \tset{P_1} = \tset{E_P} $.
	
	Second, by the hypothesis, $ P $ has a root. Let $ R $ be a root of $ P $. Set $ R^\downarrow = R \cap P^\downarrow $ and notice that $ R^\downarrow $ is a root of $ P^\downarrow $, as a prob for $\mathcal F $. In particular, $ R^\downarrow $ does not intersect any element of $ \mathcal F $. Now, let $ A \in N_{\mathcal F'} (P_2) $. By Property~\ref{prop:more-on-Gamma-and-disjoint-probs}, we have $ A \in N_{\mathcal F}(P) $. Therefore, $ R \subseteq \Ter A$. Consequently, $ R^\downarrow \subseteq \Ter A $. This implies item (1) of the definition of stable prob. Moreover, since $ A \in N_{\mathcal F}(P) $ and $ P$ is stable, we have  
	$$ \bset{A} < \bset{P} = \bset{P^\downarrow} = \bset{P_2}, \text{ and } \tset{A} > \tset{P} \geq \tset{P^\downarrow} = \tset{P_2}, $$
	which implies item (3) of the definition. Also, since $ A $ crosses $ P$ vertically, by Property~\ref{prop:crossing-a-subrectangle}, it crosses $ P^\downarrow $ vertically as well, which implies item (4) of the definition. 
	
	Now, assume that $ A, B \in N_{\mathcal F'} (P_2) $ and $ A \neq B $. Again, by Property~\ref{prop:more-on-Gamma-and-disjoint-probs}, we have $ A, B \in N_{\mathcal F}(P) $. Thus, $ A\cap B =\varnothing $, implying item (2) of the definition.  Hence, $P_2 $ is a stable prob.

	Now, we prove that $\mathcal F' $ satisfies Constraints (C1)-(C6).
	
	{\noindent \textbf{Claim}. \textit{$ \mathcal F' $ satisfies \emph{(C1).}}} 
	
	Let $ A, B \in \mathcal F'$ be two distinct and intersecting transformed copies of $ S $. Set $ L_A = \{ (x,y) \in A : x = \lset A  \} $. Notice that $ \lset A = \lset{L_A}$. 
	
	If $ A, B \in \mathcal F $, then the result holds because $ \mathcal F $ satisfies (C1). Furthermore, by Property~\ref{prop:prop-after-def-Gamma}, we cannot have $ A, B \in \mathcal F' \sm \mathcal F$. So, without loss of generality, assume $ A \in \mathcal F' \sm \mathcal F $, so $ A = S_P $ for some $ P \in \mathcal P $, and $ B \in \mathcal F$. In particular, $ B \subseteq \mathfrak B $, and by construction, $ A \cap (\mathfrak{B} \sm P^\uparrow) = \varnothing $. Hence, $ B \cap P^\uparrow \neq \varnothing $, and therefore $ B \in N_{\mathcal F}(P) $. Thus, by Property~\ref{prop:more-on-Gamma-and-disjoint-probs}, for every root $ R $ of $ P $, we have $ R \subseteq \Ter B  $. Moreover, we have $ \bset{B}< \bset{P} $ and $ \tset{B}> \tset{P} $. Also, notice that by construction, for every $ \mathfrak{s} \in \{\mathfrak{l}, \mathfrak{r}, \mathfrak{b}, \mathfrak{t} \}$, we have $ \mathfrak{s}(A) = \mathfrak{s}(P^\uparrow)$. Let $ p= (x,y) \in  L(A) $. So, $ x =\lset{A} \lset{P}$ and $ y \in (\bset{P}, \tset{P})$. Moreover, $ \bset P \leq \bset A  \leq y \leq \tset A  \leq \tset P $. Therefore, $ (x, y) \in \{ (x',y') \in P : x' = \lset P \}$. Consequently, $ (x, y) \in R $. So, $ L_A \subseteq R \subseteq \Ter B $.

	Moreover, we have: \begin{multline*} \lset B = \lset{\boxset B } \leq \lset{\Ter B} \leq \lset{L_A} \\  = \lset A  = \lset{P^\uparrow}  =  \lset P  = \lset R < \rset R  \\  \leq \rset{\Ter B} \leq \rset{\boxset B} = \rset B \overset{(a)}{\leq} \rset{\mathfrak{B}} \overset{(b)}{<} \rset{A}, \end{multline*}	 where (a) is because $ B \in \mathcal F $, and (b) follows from Step (S$'$2) of the construction. Therefore $ \lset B \leq \lset A < \rset B < \rset A $.
	
	On the other hand, \begin{equation*} \bset B < \bset P < \bset{P^\uparrow} = \bset A \overset{(c)}{<} \tset A = \tset P < \tset B,  \end{equation*} where (c) follow from the fact that $ A $, a strong Pouna set, cannot be a subset of a horizontal line segment. Therefore $ \bset B < \bset A < \tset A < \tset B$. 
	
	Hence, all the items in Constraint (C1) hold and $ A \adj B $.

	{\noindent \textbf{Claim}. \textit{$ \mathcal F' $ satisfies \emph{(C2).}}} 
	
	Let $ A $ and $ B $ be two disjoint sets in $ \mathcal F' $ such that $ A \cap \Ter B \neq \varnothing $. We prove that $ A , B \in \mathcal F $. For the sake of contradiction, assume that $ \{A, B\} \nsubseteq \mathcal F$. There are three cases possible. 
	
	Case 1: $ A, B \in \mathcal F' $. So, there exists $ P, Q \in \mathcal P $ such that $ P \neq Q $ and $ A=S_P $ and $ B = S_Q $. But in that case, by construction, $ \boxset B \subseteq Q $, and $ A \subseteq P$. So, from $ A \cap \Ter B \neq \varnothing $, we have $ P \cap Q \neq \varnothing $, a contradiction. 
	
	Case 2: $ A =S_P $ for some $P \in \mathcal P $, and $ B \in \mathcal F $. Since $ A \subseteq P $, form $  A \cap \Ter B \neq \varnothing $ we deduce that $ P \cap \Ter B \neq \varnothing$. Choose $ p  = (x,y) \in P \cap \Ter B $. Because by definition of Territory, there exists a point $p'= (x',y) \in B $ with $ x' > x $. Now, because $ B \subseteq F $, we have $ p' \in \mathfrak B $ and therefore $ p' \in P $. Hence $ P \cap B \neq \varnothing $, i.e.\ $ B \in N(P) $. Therefore, $ B $ crosses $ P$ vertically. Moreover, $A = S_P $ crosses $ P_1 $ and therefore $P$ horizontally. So, by \Cref{prop:horizont-and-vertical-crossing-intersect}, we have $ A \cap B \neq \varnothing $, a contradiction.
	
	Case 3: $ A \in \mathcal F $ and $ B =S_P $ for some $P \in \mathcal P $. In this case $ \Ter B \subseteq P $, and therefore $ A \cap P \neq \varnothing $, i.e. $A \in N(P) $. So, $ A $ crosses $P $ vertically. On the other hand, $ B $ crosses $ P_1 $ and thus $ P$ horizontally. Therefore, by \Cref{prop:horizont-and-vertical-crossing-intersect}, we have $ A \cap B \neq \varnothing $, a contradiction.

	{\noindent \textbf{Claim}. \textit{$ \mathcal F' $ satisfies \emph{(C3).}}} 
	
	Let $ A , B \in \mathcal F $ be two distinct sets with non-empty intersection. For the sake of contradiction, assume that there exists $ C \in \mathcal F $ such that $ C \subseteq \Ter A \cap \Ter B $. We first show that $ C \in \mathcal F $. Suppose not, so $ C  = S_P $ for some $P  \in \mathcal P$. Since $ C \subset P$, neither of $ A $ and $ B $ can be some set of the form $ S_Q $. Therefore $ A, B \in \mathcal F $. Now, notice that $ C \subseteq \Ter A \subseteq \boxset A$. On the other hand, $ \boxset A \subseteq \mathfrak B $, but $ C \nsubseteq \mathfrak B$, a contradiction. 
	
	Now we prove that both $ A $ and $ B $ are in $ \mathcal F $. Suppose not. Without loss of generality, assume that $A= S_P $ for some $P \in \mathcal P $. Since $ C \subseteq \Ter A $, we must have $ C \in N(P) $. Therefore $ \bset C < \bset P \leq \bset A $. On the other hand, because $ C \subseteq \Ter A \subseteq \boxset A $, we have $ \bset C \geq \bset A $, a contradiction. So, $ A, B \in \mathcal F $ as well, and the result follows from the fact that $\mathcal{F}$ satisfies~(C3).

	{\noindent \textbf{Claim}. \textit{$ \mathcal F' $ satisfies \emph{(C4).}}} 
	
	Fix $P \in \mathcal P$. Let us first prove that there exists no $ A \in \mathcal F; $ such that $ A \adj S_P $ or $ S_P \prec A $. First, if $ A \adj S_P$, then, in particular, $ A \cap S_P \neq \varnothing $. Thus, by Property~\ref{prop:more-on-Gamma-and-disjoint-probs}, we have $ A\in \mathcal F $. Therefore, $ \rset{A} \leq \rset{\mathcal F} < \rset{P^\uparrow} = \rset{S_P}$. But on the other hand, $ A \prec S_P $ implies $ \rset{A} > \rset{S_P}$, a contradiction. Second, if $ S_P \prec A $, then in particular $S_P \subseteq \Ter{A} $. Also, by construction $ S_P \subseteq P $. Therefore, $ \Ter{A} \cap P \neq \varnothing $. Hence, by Property~\ref{prop:more-on-Gamma-and-disjoint-probs}, we have $ A \in \mathcal F $. Therefore $ \rset{A} \leq \rset{\mathcal F} < \rset{P^\uparrow} = \rset{S_P}$. On the other hand, by Lemma~\ref{lem:prec-properties}, $ S_P \prec  A $ implies that $ \rset{S_P} < \rset{A}$, a contradiction.
	
	Now, for the sake of contradiction, assume that there exists $ A, B, C  \in \mathcal F' $ such that $ A \prec B $, $ A \adj C $, and $ B \adj C $. From what we proved above, we know that $ A, C \in \mathcal F$. Therefore, since $ \mathcal F $ satisfies (C4), we cannot have $ B\in \mathcal F $. So, $ B = S_P $ for some $ P \in \mathcal P $. In particular $ \Ter{B} \subseteq \boxset{B} \subseteq P^\uparrow$. 
	
	From $A \prec B $, we have $ A \subseteq \Ter B \subset P^\uparrow \subseteq P $. Therefore, $ A \in N_{\mathcal F}(P)$. 
	
	On the other hand, from $ B \prec C $, we have $ B \cap C \neq \varnothing $, therefore $ C \cap P^\uparrow \neq \varnothing $. So, $ C \in   N_{\mathcal F}(P)$. 
	
	So, $ A $ and $ C $ are two sets in $ N_{\mathcal F}(P) $ that are not disjoint, which contradicts the fact that $ P$ is stable.

	{\noindent \textbf{Claim}. \textit{$ \mathcal F' $ satisfies \emph{(C5).}}} 
	
	For the sake of contradiction, assume that $ A $, $ B $, and $ C $ are three sets in $ \mathcal F' $ that two by two intersect. At least one of the three sets must be in $ \mathcal F' \sm \mathcal F $, because (C5) holds for $\mathcal F $. Moreover, because of Property~\ref{prop:more-on-Gamma-and-disjoint-probs}, at most one of the three sets is in $ \mathcal F' \sm \mathcal F $. So, without loss of generality, assume that $ A=S_P $ for some $P \in \mathcal P $, and that $ B, C \in \mathcal F $. But since $ B \cap A \neq \varnothing $, we have $ B \cap P \neq \varnothing$, i.e.\ $B \in  N(P) $. Similarly, $ C \in N(P)$. But $ B \cap C \neq \varnothing $ contradicts the fact that $ P$ is stable for $ \mathcal F $. Hence, (C5) holds for $ \mathcal F' $.

	{\noindent \textbf{Claim}. \textit{$ \mathcal F' $ satisfies \emph{(C6).}}} 
	
	By assumption, $ S $ is strong. So, it is enough to show that $ T_P $, in Step (S$'$2), is a positive transformation for every $P \in \mathcal{P}$. This follows from the fact that $ T_p $ is positive, as shown in Property~\ref{prop:prop-after-def-Gamma}. 
	
	This completes the proof of the lemma.  
\end{proof}

\medskip
Now we can prove that Burling graphs are constrained $ S $-graphs. 

\begin{theorem} \label{thm:equalthm-Burling-is-constraintS}
	Let $ S $ be a Pouna set. Every Burling graph is a constrained $ S $-graph. 
\end{theorem}
\begin{proof} For this proof, we adopt the notations in the definition of the construction of Pawlik, Kozik, Krawczyk, Laso\'n, Micek, Trotter, and Walczak.
	
	We may assume that $ S $ is a strong Pouna set, otherwise, we replace every $ S $ in this proof with the horizontal reflection of $ S $. Fix a subterritory $ E $ of $ S $ (which exists, by Lemma~\ref{lem:subterritory-exists}), and apply the construction on it. For every $ k \geq 1 $, we know that $ \mathcal F_k $ is a family of transformed copies of  $ S $. We first prove that $ \mathcal F_k$ satisfies Constraints (C1)-(C6). To do so, we prove the following stronger statement by induction on $ k$.

	\begin{statement} For every $ k \geq 1 $, we have: \begin{enumerate} \item the elements of $ \mathcal P_k $ are mutually disjoint, \item $\mathcal P_k $ is a family of stable probs of $ \mathcal F_k $,  \item $\mathcal{F}_k$ satisfies constraints (C1)-(C6). \end{enumerate} \end{statement}
	
	First of all, for $ k=1 $, the first item of the statement follows from the fact that the fact  that $ E $ is a subterritory of $ S $. Statement (2) and (3) hold trivially, as $ |\mathcal{F}_1|=1 $.

	Now, assume that the statement holds for some $ k \geq 1 $, we prove that it holds for $ k+1 $. 
	
	Notice that for every $ P \in \mathcal P$, the transformation $ T'_P$ is positive, so the tuple $ (\mathcal F^P, \mathcal P^P) $ in a positive transformed copy of $ \Gamma(\mathcal F_{k}, \mathcal P_{k}) $. So, by Property~\ref{prop:T(F)-satisfies-C1-C5}, we know that  \begin{equation} \label{eq:FP-sat-C1-6} \text{for every $ P \in \mathcal P $, the family $ \mathcal F^P $ satisfies Constraints (C1)-(C6).} \end{equation} Moreover, it is easy to check the following: \begin{equation} \label{eq:P-stable-mut-disj} \text{for every $ P \in \mathcal P $, the elements of $ \mathcal P^P $ are stable probs for $ \mathcal F$ and are mutually disjoint.} \end{equation}

	{\noindent \textbf{Claim}. \textit{The elements of $ \mathcal P_k $ are mutually disjoint.}} 
	
	Let $ P_Q $ and $P'_{Q'} $ be two probs in $ \mathcal P_{k+1} $. In order to show that these two probs are disjoint, it is enough to show that $ (\bset{Q}, \tset{Q}) $ and $(\bset{Q'}, \tset{Q'}) $ are disjoint intervals. If $ P = P' $, then this follows from (\ref{eq:P-stable-mut-disj}), and if $ P \neq P' $ from the fact that $ Q $ and $ Q' $ are inside the roots of $P$ and $ P'$ respectively, and $ P $ and $ P' $ are disjoint by induction hypothesis.

	{\noindent \textbf{Claim}. \textit{Every $ P \in \mathcal P_{k+1} $ is a stable prob for $\mathcal F_{k+1}$.}} 
	
	Let $ P_Q  \in \mathcal P_{k+1} $. Notice that $ Q \in \mathcal P^P$ is a prob for $\mathcal F^P $. So, by (\ref{eq:P-stable-mut-disj}), $ Q $ has a root $ R $ such that for every $ A \in N_{\mathcal F^P}(Q)$, we have $ R \subseteq \Ter{A}$. So, item (1) of the definition of stable prob holds. 
	
	Set $ N_1 = N_{\mathcal F^P}(Q) $ and $ N_2 = N_{\mathcal F_{k+1}}(P) $. 
	
	The elements in $ N_{\mathcal F_{k+1}}(P_Q) $ are either the neighbors of $ Q $ as a prob for $ \mathcal F^P $, so they are in $N_1  $, or are outside $ R_P $ and thus are in $N_2 $. The elements in $N_1 $ are mutually disjoint by (\ref{eq:P-stable-mut-disj}) and the elements in $ N_2 $ are mutually disjoint by induction hypothesis. Finally, one element in $ N_1 $ and one element in $N_2 $ are disjoint because the former is inside $ R_P $ and the latter does not intersect $ R_P $. So, item (2) of the definition holds as well.
	
	Now, fix $ A \in N_{\mathcal F_{k+1}}(P_Q)$. If $ A \in N_1 $, then $$ \bset{A} < \bset{Q} = \bset{P_Q}, \text{ and } \tset{A} > \tset{Q} = \tset{P_Q}.  $$ Moreover, there is a path in $ A $ crossing $ Q $. So, the same path crosses $P_Q $ as well. 
	
	If $ A \in N_2 $, then $$ \bset{A} < \bset{P} = \bset{R_P} \leq \bset{Q} = \bset{P_Q},	 $$ and $$ \tset{A} > \tset{P} = \tset{R_P} \geq \tset{Q} = \tset{P_Q}. $$ Moreover, there is a path in $ A $ crossing $ P$, so by \Cref{prop:crossing-a-subrectangle}, it crosses $P_Q $ as well. 
	
	Now, we check that $\mathcal{F}_{k+1}$ satisfies Constraints (C1)-(C6). In what follows, we use several times the fact that that by (\ref{eq:FP-sat-C1-6}) and by induction hypothesis, the conditions hold when all the elements are chosen inside $ \mathcal F_k $ or inside $ \mathcal F^P $ for some $P \in \mathcal P_k$. 
	
	Moreover, notice that by induction hypothesis, elements of $ \mathcal P_k $ are disjoint. Now, because every $A \in \mathcal F^P $ is entirely inside $ P $, we know that  \begin{equation} \label{eq:FP-FQ-disjoint} \text{if $ P \neq Q $, then the elements of $ \mathcal F^P $ are disjoint from the elements of  $F^Q$.} \end{equation}
	
	Furthermore, for every $ P \in \mathcal P_k$, the elements of $ \mathcal F^P $ are all inside $ R_P$. Moreover, by definition of root, no element of  $\mathcal F_k $ intersect $ R_P $, so, \begin{equation} \label{eq:Fk-FP-disjoint} \text{for every $ P\in \mathcal P $, the elements of $ \mathcal F_k $ are disjoint from the elements of $ \mathcal F^P $. } \end{equation}

	{\noindent \textbf{Claim}. \textit{$ \mathcal F_{k+1}$ satisfies \emph{(C1).}}}

	Let $ A, B \in \mathcal F_{k+1} $ be two distinct elements such that $ A \cap B \neq \varnothing $. By (\ref{eq:FP-FQ-disjoint}) and (\ref{eq:FP-FQ-disjoint}), either $ A, B \in \mathcal F_k $ or there exists $ P \in \mathcal P_k$ such that $ A, B \in \mathcal F^P $. In the former case, by induction hypothesis, we have $A \adj B $ or $ B \adj A $. In the latter case, by (\ref{eq:FP-sat-C1-6}), we have $ A\adj B $ or $ B \adj A $.

	{\noindent \textbf{Claim}. \textit{$ \mathcal F_{k+1}$ satisfies \emph{(C2).}}} 
	
	Let $ A, B \in \mathcal F_{k+1} $ such that $ A \cap B = \varnothing $ and $ A \cap \Ter{B} \neq \varnothing $. There are four cases possible:
	
	Case 1: $ A, B \in \mathcal F_k $, in which case the result follows from (\ref{eq:FP-sat-C1-6}). 
	
	Case 2: $ A \in \mathcal F_k $ and $ B \in \mathcal F^P $ for some $P \in \mathcal P_k $. 
	
	This case is not possible, because $ \Ter{B} \subseteq \boxset{B} \subseteq R_P $. However, $ A \in \mathcal F_k $, so $ A $ does not intersect $ R_P $ as it is a root of a prob for $ \mathcal F_k $.

	Case 3: $ A \in \mathcal F^P $ for some $P \in \mathcal P_k $ and $ B \in \mathcal F_k $.
	
	Since $ A \subseteq R_P $, we have $ R_P \cap \Ter{B} \neq \varnothing $. Let $ p = (x,y) \in R_P \cap \Ter{B}$. By the definition of territory, there exists $ x' > x $ such that $ p =(x', y) \in B $. Moreover, since $ R_P $ is a root of $ P $, we have $ p' \in P $. So, $ p' \in B \cap P $. Therefore, $ B \in N_{\mathcal F_k}(P)$. Hence, by (\ref{eq:P-stable-mut-disj}) and using Property~\ref{prop:every-root-stable}, we have that every root of $ P $ is inside the territory of $ B$. Hence, $ A \subseteq R_P \subseteq \Ter{B} $.  So, the result holds. 
	
	Case 4: $ A \in \mathcal F^P $ and $ B \in \mathcal F^Q $ for $ P, Q \in \mathcal P_k $. Let $ p \in A \cap \Ter{B} $. So, in particular $ p \in A \subseteq P $ and $ p \in \Ter{B} \subseteq \boxset{F^Q} \subseteq Q $. Therefore $ P\cap Q \neq \varnothing $. Hence, by the induction hypothesis, we must have  $P = Q $. So, the result follows from (\ref{eq:FP-sat-C1-6}).

	{\noindent \textbf{Claim}. \textit{$ \mathcal F_{k+1} $ satisfies \emph{(C3).}}}

	Let $ A, B \in \mathcal F_{k+1} $ be two distinct sets such that $ A \cap B \neq \varnothing $. For the sake of contradiction, assume that there exists $ C \in \mathcal F_{k+1} $ such that $ C \subseteq \Ter{A} \cap \Ter{B} $. 
	
	First of all, by (\ref{eq:FP-FQ-disjoint}) and (\ref{eq:Fk-FP-disjoint}), there are only two possible cases for $ A $ and $ B $: either $ A, B \in \mathcal F_{k} $ or $ A, B \in \mathcal F^P $ for some $ P \in \mathcal P_k $. 
	
	Case 1: $ A, B \in \mathcal F_{k} $. In this case, by induction hypothesis, we cannot have $ C \in \mathcal F_k $.  So, $ C \in \mathcal F^P $ for some $ P \in \mathcal P_k $. Consequently, $ C \subseteq R_P $. Now, let $ p = (x,y) \in C $. Since $ C \subseteq \Ter{A} $, there exists $ x'>x $ such that $ p'=(x',y) \in A $. But also, $ p' \in P $. Therefore $ A \in N_{\mathcal F_{k}}(P)$. Similarly, we can show that $ B \in N_{\mathcal F_k}(P)$.  A contradiction with the fact that the elements in $ N_{\mathcal F_k}(P) $ are mutually disjoint. 
	
	Case 2: $ A, B \in \mathcal F^P $ for some $ P \in \mathcal P_k $. Notice that  $$ \Ter{A} \subseteq \boxset{A} \subseteq \boxset{\mathcal F^P} \subseteq R_P. $$ So, $ C \subseteq R_P $. Therefore $ C \in \mathcal F^ P$ as well, and the result follows from (\ref{eq:FP-sat-C1-6}).

	{\noindent \textbf{Claim}. \textit{$ \mathcal F_{k+1} $ satisfies \emph{(C4).}}} 
	
	Assume, for the sake of contradiction, that there exists $ A, B, C \in \mathcal F_{k+1} $ such that  $ A\prec B $, $ A \adj C $, and $ B \adj C $. By (\ref{eq:FP-FQ-disjoint}) and (\ref{eq:Fk-FP-disjoint}), since $ A \cap C \neq \varnothing $ and $ B \cap C \neq \varnothing $, either $ A, B, C \in \mathcal F_k $ or $ A, B, C \in \mathcal F^P $ for some $P \in \mathcal P_k $. The former is not possible because of induction hypothesis, and the latter because of (\ref{eq:FP-sat-C1-6}). So, there exist no such triple.

	{\noindent \textbf{Claim}. \textit{$ \mathcal F_{k+1} $ satisfies \emph{(C5).}}} 
	
	For the sake of contradiction, assume that there exist three distinct sets $ A, B, C \in \mathcal F_{k+1}$ that are mutually intersecting. By induction hypothesis, such triple does not exists in $ \mathcal F_{k}$. So, at least on of the sets is in $ \mathcal F^P $ for some $P \in \mathcal P_k $. But then, (\ref{eq:FP-FQ-disjoint}) and (\ref{eq:Fk-FP-disjoint}) imply that the three sets are all in $\mathcal{F}^P$, a contradiction with (\ref{eq:FP-sat-C1-6}).

	{\noindent \textbf{Claim}. \textit{$ \mathcal F_{k+1} $ satisfies \emph{(C6).}}} 	 By assumptions, $ S $ is strong. Thus, we only need to show that every element of $ \mathcal F_{k+1}$ is a positive transformed copy of $ S $. This is true since the elements of $ \mathcal F_k $ are positive transformed copies of $ S $ and the elements of each $ \mathcal F^P $ are also positive transformed copies of $ S$, because by (\ref{eq:FP-sat-C1-6}), the family $\mathcal F^P $ satisfies (C6).

	This finishes the proof of the statement. Therefore, for every $ k \geq 1$, the oriented intersection graph of $\mathcal{F}_k$ is a constrained $ S $-graph. 
	To complete the proof we remind that the class of graphs generated by the oriented intersection graphs of $ \mathcal F_k$'s is exactly the class of Burling graphs.
\end{proof}

\subsection{Concluding the proof}

Let us conclude the proof of \Cref{thm:main-complete-theorem}.

\begin{proof}[Proof of \Cref{thm:main-complete-theorem}]
	Let $ S $ be a Pouna set. Let $ G $ be a graph. If $ G $ is a constrained graph, then by \cref{thm:equalthm-constrained-implies-abstract-Burling}, it is a Burling graph. If $ G $ is a Burling graph, then by \cref{thm:equalthm-Burling-is-constraintS}, it is a constrained $ S $-graph. Finally, if $ G $ is a constrained $ S $-graph, then by definition, it is also a constrained graph.
\end{proof}

\section*{Acknowledgments}

I would like to thank Frédéric Meunier for pointing out \Cref{thm:flores-van-kampen} to me, Paul Meunier for useful discussions and his contributions to some proofs, in particular Lemmas~\ref{prop:horizont-and-vertical-crossing-intersect} and~\ref{lem:subterritory-exists}, Gael Gallot for fruitful discussions, in particular during his internship on the topic, Julien Duron for reading some parts of the first version of this work, and J\'er\'emie Chalopin for his useful comments on related results in my Ph.D.\ thesis.

\appendix

\section{Other proofs} \label{appendix-proofs}

\begin{proof}[\textbf{Proof of \Cref{lem:real-border-lemma}}] Notice that $$ B = [B \cap A^\circ]  \cup [B \cap (X\sm \bar{A})]   \cup [B \cap \partial A]. $$ The sets, $ B \cap A^\circ $ and $ B \cap (X\sm \bar{A})$ are both open in $B $ and each is non-empty by the assumption. Moreover, their intersection is the empty set. So, if $ B \cap \partial A \neq \varnothing $, then $ B $ can be written as the union of two non-empty and non-intersecting sets that are open in $ B$, and thus $ B $ is not connected.   
\end{proof}

\begin{proof}[\textbf{Proof of \Cref{lem:crossing-between-two-lines}}] 
	Let $ X_0 = \gamma^{-1}(L_0) = \{ x\in [0,1] : \gamma(x) \in L_0 \} $. Notice that $ X_0 $ is closed since it is the pre-image of a closed set under a continuous function, and is bounded. So, $ X_0 $ is compact. Moreover, $ 0 \in X_0 $, so $ X_0 \neq \varnothing $. Thus, we can set $ x_0 = \max X_0$. 
	
	Set $ \gamma'' = \gamma|_{[x_0,1]} $, and let $ X_1 = \gamma''^{-1}(L_1) = \{x \in [x_0, 1]: \gamma''(x) \in L_1 \}$. Again, $ X_1 $ is compact, and it is non-empty since $ 1 \in X_1 $. So, we can set $ x_1 = \min X_1 $. 
	
	Set $ \gamma' = \gamma''|_{[x_0, x_1]}$. We prove that $ \im{\gamma'} \subseteq \{ (x,y) : y_0 \leq y \leq y_1 \} $.
	
	Assume, for the sake of contradiction, that there exists a point $ t \in (x_0, x_1) $ such that $(\pi_2 \circ \gamma'') (t) \leq y_0 $ or $(\pi_2 \circ \gamma'') (t) \geq y_1 $. In the former case, by the intermediate value theorem, there exists $ t' \geq t > x_0 $ such that $ (\pi_2 \circ \gamma'') (t) = y_0 $. Thus $ t' \in X_0$, contradicting the choice of $x_0 $. In the latter case, there exists $ t' \leq t < x_1 $ such that $ (\pi_2 \circ \gamma'') (t) = y_0 $. Thus $ t' \in X_1$, contradicting the choice of $ x_1$.  
	
	The second point is clear from the choice of $ x_0$ and $ x_1$. 
\end{proof}

\begin{proof}[\textbf{Proof of \Cref{prop:crossing-a-subrectangle}}] 
	Let $ \gamma: [0,1] \rightarrow R \cap A $ be the crossing path. By two times use of the intermediate theorem on the function $\pi_2 \circ \gamma $, we conclude that there exist $ x_0 $ and $ x_1 $ with $ x_0 \leq x_1 $ such that $ \gamma(x_0) $ and $ \gamma(x_1) $ are respectively on the bottom side-and the top-side of $ R' $. Applying \Cref{lem:crossing-between-two-lines} to the path $ \gamma|_{[x_0, x_1]} $ completes the proof of the lemma.  
\end{proof}

\subsection*{For the proof of \Cref{prop:horizont-and-vertical-crossing-intersect}}

We recall that an \emph{arc}\index{arc} in a topological space $ X $ is a homeomorphism from a closed interval in $\RR$ to $ X $. In particular, every arc is a path. We say that two paths $\gamma_1 : [a_1,b_1] \rightarrow \RR$ and $\gamma_2 : [a_2, b_2] \rightarrow \RR $ are internally disjoint if their images do not intersect but possibly on common endpoints, i.e.\
for $ i \in \{1,2\}$, we have $ \gamma_i((a_i, b_i)) \cap \im{\gamma_{3-i}} = \varnothing$. 

It is well known that $K_5$ is not a planar graph. In other words, if we have 5 distinct points in the plane and every two distinct points among them are joined by an arc, then at least two of these arcs are not internally disjoint. However, it is possible to replace the ``arc" in the above statement with ``path". This fact follows from the Flores-Van Kampen theorem. In the following presentation of this theorem from~\cite{Sarkaria1991}, $ \sigma_k^d$ denotes the $k$-skeleton of the $d$-dimensional simplex. 

\begin{theorem}[Flores-Van Kampen theorem; Flores~\cite{Flores1932}; Van Kempen~\cite{VanKampen1933}] \label{thm:flores-van-kampen}
	For any continuous map $f: \sigma_{s-1}^{2s} \to \RR^{2(s-1)}$ there exist a pair $(s_1, s_2)$ of disjoint simplices of $ \sigma_{s-1}^{2s} $ such that $f(s_1)\cap f(s_2) \neq \varnothing$. 
\end{theorem}

Applying \Cref{thm:flores-van-kampen} to $ s=2 $ results in the desired statement as follows.

\begin{corollary} \label{cor:cor-of-flores-van-kampen}
	Let $ S $ be a set of 5 distinct points in the plane such that for every $ a, b \in S$ with $ a \neq b$, there exists a path $ \gamma_{a,b} $ joining $ a $ to $ b $. Then, there are four distinct points $ a, b, c, d \in S $ such that $ \im{\gamma_{a,b}} \cap \im{\gamma_{c,d}} \neq \varnothing $.
\end{corollary}

We believe that the following proof is folklore, but we could not find a reference for it.

\begin{proof}[\textbf{Proof of \Cref{prop:horizont-and-vertical-crossing-intersect}}]
	Assume, for the sake of contradiction, that $ \im{\alpha} \cap \im{\beta} = \varnothing $. Set $ a_0 = \alpha(0) $, $ a_1 = \alpha(1) $, $b_0 = \beta(0)$, and $b_1 = \beta(1)$. 
	Fix a real number $ \epsilon > $. Let $\gamma_1 $, $ \gamma_2 $, $\gamma_3$, and $\gamma_4$ be paths that respectively join $ b_1 $ to $ a_1 $, $ a_1 $ to $b_0$, $ b_0 $ to $ a_0 $, and $ a_0 $ to $b_1 $ such that every two paths among them are internally disjoint, the image of each of them is entirely outside $ R $ except for its endpoints, and all of them are entirely inside the rectangle $$ R' = [\lset{R} - \epsilon, \rset{R}+\epsilon] \times [\bset{R}-\epsilon, \tset{R} + \epsilon]. $$  See \Cref{fig:K5-planar}.
	
	\begin{figure} 
		\centering 
		\vspace*{-1cm} 
		\includegraphics[width=9.5cm]{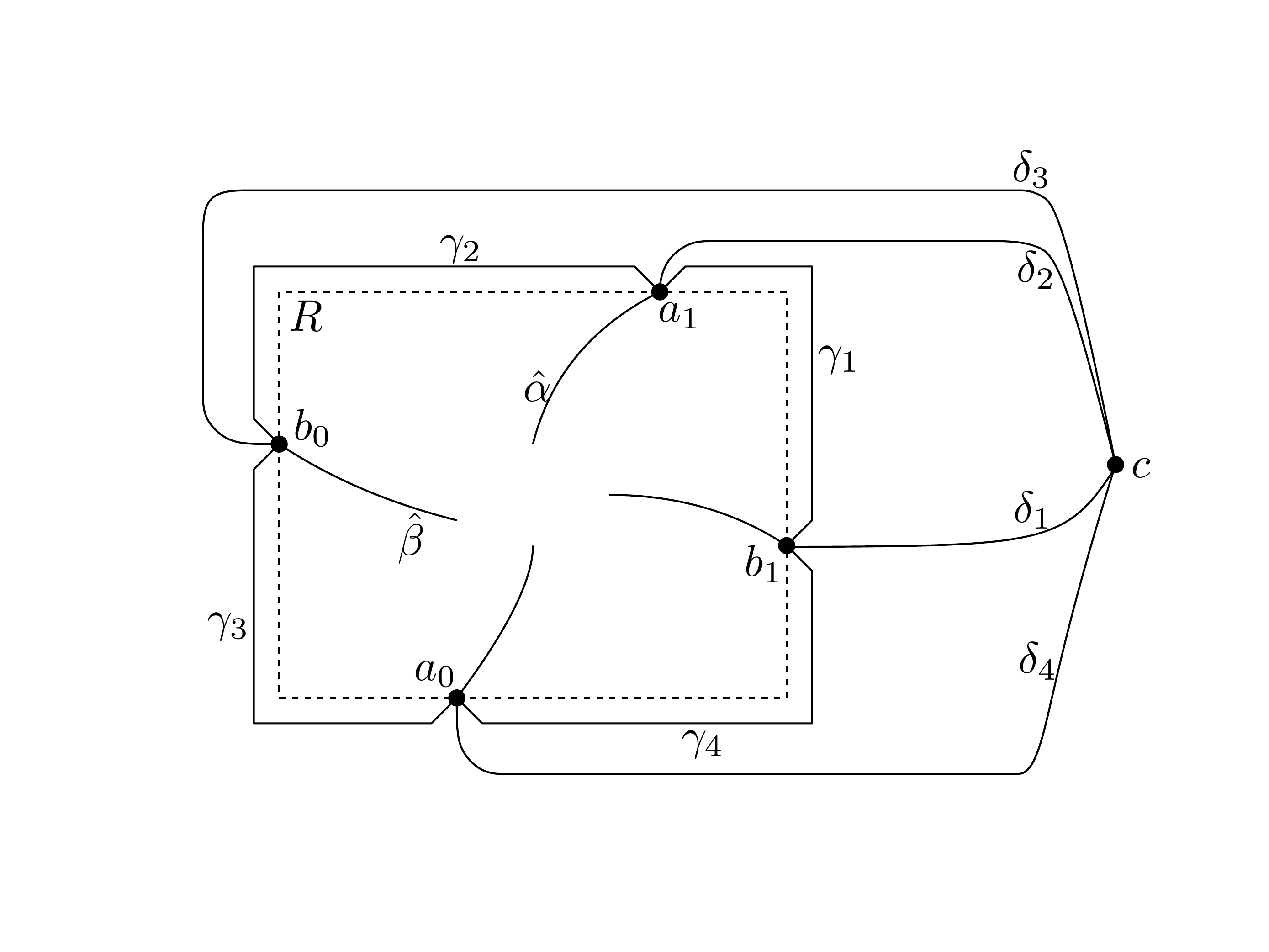} 
		\vspace{-1.2cm} 
		\caption{Proof of Lemma~\ref{lem:crossing-between-two-lines}: a planar embedding of $K_5 $.} \label{fig:K5-planar} 
	\end{figure}
	
	Finally, choose a point $ c $ outside $ R' $, and let $ \delta_1 $, $ \delta_2$, $ \delta_3$, and $ \delta_4 $ be four paths from $ c $ to $ b_1 $, $ a_1 $, $b_0 $, and $ a_0 $ respectively. Choose $\delta_i$'s such that they are two-by-two internally disjoint, and such that for each $ i, j \in \{1,2,3,4\}  $, the two paths $ {\gamma_i} $ and $ {\delta_j} $ are also internally disjoint. 
	
	The existence of the set $ S = \{a_0,a_1, b_0, b_1 , c\}$ of points and the paths $ \{ \hat{\alpha}, \hat{\beta}, \gamma_i, \delta_i : i \in \{1,2,3,4\} \} $ contradicts \Cref{cor:cor-of-flores-van-kampen}. 
\end{proof}

\begin{proof}[\textbf{Proof of \Cref{prop:ter(T)=T(ter)}}]
	Let $ T: (x,y) \mapsto (ax+c, bx+d)$. Denote the inverse of $ T$ by $ T^{-1}$. 
	
	If $ (x,y) \in \boxset{S}\sm S $, then  $$T(x,y) \in \boxset{S}\sm S = T(\boxset{S}) \sm T(S) = \boxset{T(S)} \sm T(S). $$
	
	Moreover, $ x'> x $ implies $ ax'+b> ax+b$. Therefore, $(x,y) \in \Ter{S}$ implies $ T(x,y) \in \Ter{T(S)} $. Hence, $ \Ter{S} \subseteq \Ter{T(S)}$.
	
	To finish the proof, notice that $S = T^{-1}(T(S)) $ and $T^{-1} $ is also a positive transformation. Thus, by what precedes, $\Ter{T(S)} \subseteq \Ter{S}$. 
\end{proof}

\begin{proof}[\textbf{Proof of \Cref{prop:T(E)-subter-of-T(S)}}] 
	Set $ S' = T(S)$ and $E' = T(E)$. We prove that the three items of the definition hold and $ E' $ is a subterritory of $ S' $. 
	
	
	First, by \Cref{prop:ter(T)=T(ter)}, we have that $ E'=T(E) \subseteq T(\Ter{S}) = \Ter{S'} $. So the first item of the definition holds.
	
	Second, since $ a$ and $ b $ are positive, for every compact set $ A $ we have  
	\begin{equation*} 
		\begin{split} \lset T(A) &= \min\{x: (x,y)\in T(A) \} \\ &= \min \{x: \big(\frac{x-c}{a}, \frac{y-d}{b}\big) \in A \} \\ & = \min \{ au+c : (u,v) \in E \} = a . \lset{A} + c.   
		\end{split} 
	\end{equation*} 
	In the equations above we have again used the change of variables $ u = \frac{x-c}{a} $ and $ v = \frac{y-d}{b} $.  So,  $$ \lset{E'} = a. \lset{E} + c < a. \lset{S} + c = \lset{S'}. $$ 
	
	The proof of the rest of the inequalities is similar.  This proves the second item. 
	
	Finally, let $ P$ be the prob for $ \boxset{S} $ defined by $ E $ and let $ \gamma: [0,1] \rightarrow S \cap P $ be the path connecting the top-side of $ P$ to the bottom-side of $ P$. Denote by $P'$ the prob for $ \boxset{S'} $ defined by $ E' $. Notice that $ P'=T(P)$. So, $T(S \cap P)= T(S) \cap T(P) = S' \cap P' $. Thus, the function $ T\circ \gamma: [0,1] \rightarrow S' \cap P' $ is a path entirely inside $ S' \cap P' $. Moreover, since $ T$ sends the top-side (resp.\ bottom-side) of $ P$ to the top-side (resp.\ bottom-side) of $ P' $, we have that $ (T \circ \gamma) (0) $ is on the top-side of $ P' $ and $ (T\circ \gamma)(1) $ is on the bottom-side of $ P'$, and this finishes the proof of the third item. 
\end{proof}

\begin{proof}[\textbf{Proof of \Cref{prop:T(F)-satisfies-C1-C5}}]
	Set $ F' = \{T(S): S \in \mathcal F\} $. Suppose that $T: (x,y) \mapsto (ax+c, by+d) $ where $ a > 0 $ and $ b> 0$. 
	
	First of all, notice that $ A \cap B \neq \varnothing $ if and only if $ T(A) \cap T(B) \neq \varnothing $. So, two sets $T(A)$ and $T(B)$ in $ \mathcal F' $ intersect if and only if $ A $ and $ B $ intersect in $ F $. 
	
	Second, notice that for every set $ A $, $ \lset{T(A)} = a .\lset{A} +c $. So, since $ a> 0$, if $ \lset{A} \leq \lset B$, then $\lset{T(A)} \leq \lset{T(B)}$.  
	
	Third, if $ A \subseteq B $, then $ T(A) \subseteq T(B) $, because if $ p \in  T(A) $, then $ p = (ax+c,by+d) $ for some $ (x,y) \in A $. Now, since $ (x,y) \in B $, we have $ p \in T(B) $.  
	
	Fourth, notice that $ \Ter{T(A)} = T(\Ter{A}) $. This, along with the third fact implies that if $ A \subseteq \Ter{B} $, then $T(A) \subseteq \Ter{T(B)}$. 
	
	With the four facts above, it is easy to check that $ \mathcal  F' $ satisfies Constraints (C1)-(C6). 
\end{proof}

\begin{proof}[\textbf{Proof of \Cref{prop:prop-after-def-Gamma}}]
	The proof of (1) is immediate from the definition of  $ T_P $.
	
	To prove (2), set $ T_P: (x,y) \mapsto (ax+c, bx+d)$. We have  $$ a = \frac{2\wset{S}}{\lset{E} - \lset{S}}.\frac{\wset{P^\uparrow}}{\wset{S}}, $$ and $$  c =  \frac{2 \wset S}{\lset E - \lset S} \big(\lset{P^\uparrow} - \frac{\lset{S}\wset{P^\uparrow}}{\wset{S}}\big) + \lset{P^\uparrow}(1 - \frac{2 \wset{S}}{\lset E - \lset S}). $$ 
	
	Now, notice that  \begin{equation*} \begin{split} \lset{T_P(E)} &= a.\lset{E} + c \\ &= \lset{E}.\frac{2 \wset{S} \wset{P^\uparrow}}{\wset{S}(\lset{E}-\lset{S})} + \frac{2\wset{S}\lset{P^\uparrow}}{\lset{E}-\lset{S}} \\ &  \ \ \ - \lset{B_S}.\frac{2 \wset{S} \wset{P^\uparrow}}{\wset{S}(\lset{E}-\lset{S})} + \lset{P^\uparrow} - \frac{2\wset{S}\lset{P^\uparrow}}{\lset{E}-\lset{S}} \\ &= \lset{P^\uparrow} + (\lset{E}-\lset{S})\frac{2 \wset{S} \wset{P^\uparrow}}{\wset{S}(\lset{E}-\lset{S})} \\ & > \lset{P^\uparrow} + 2 \wset{P^\uparrow}  =  \rset{P^\uparrow} + \wset{P^\uparrow}  > \rset{P^\uparrow}. \end{split} \end{equation*} 
	To complete the proof, notice that $ \rset{P^\uparrow} = \rset{\boxset{\mathcal F}} $.  
\end{proof}

\begin{proof}[\textbf{Proof of \Cref{prop:more-on-Gamma-and-disjoint-probs}}] 
	Item (1) follows from the facts that $ \boxset{S_P} \subseteq P $, $ S_P \subseteq P$, $ S_Q \subseteq Q$, and $ P \cap Q = \varnothing $.
	
	To prove (2), notice that by \Cref{prop:prop-after-def-Gamma}, we have $ \lset{E_P} > \rset{\boxset{\mathcal F}}$. Since $ P_1 $ is the prob defined by $ E_P$, the prob $ P_1 $ is also outside $ \boxset{\mathcal F}$. So, for every $ A \in \mathcal F$, we have $ A \notin N_{\mathcal F'}(P_1)$. Moreover, by item (1) of this property, for every $ Q \in \mathcal P \sm \{P\}$, we have $ S_Q \notin N_{\mathcal F'}(P_1) $. Finally, since $ E_P $ is a subterritory of $ S $, by definition of $ S_P \cap P_1 \neq \varnothing $. Therefore $ N_{\mathcal F'}(P_1) = \{S_P\}$.

	To prove (3), assume that $ A \in \mathcal F' $ is of the form $ A = S_Q $ for some $ Q $. Case 1, $ Q = P$, in which case $S_Q = S_P \subseteq P_1 $, and since $P_1 \cap P_2 = \varnothing $, we have $ A \notin N_{\mathcal F'}(P_2) $. Case 2, $ Q \neq P $, and thus item (1) of this property implies that $ A \notin N_{\mathcal F'}(P_2)$. Therefore, $ N_{\mathcal F'}(P_2) \subseteq \mathcal F$.
	
	Hence, $ N_{\mathcal F'}(P_2) = N_{\mathcal F}(P_2)$. So, since $ P_2 \cap \boxset{F} \subseteq P $, we have $ N_{\mathcal F'}(P_2) \subseteq N_{\mathcal F}(P)$.
	
	To prove (4), first notice that $ \rset{S_P}> \rset{\mathcal F}$, along with (1), imply that there exists no $ B \in \mathcal F'$ such that $ B \adj S_P $. Now, set $ N(S_P) $ to be the set $ \{A\in \mathcal F': S_P \adj A \}$. by construction, $ N(S_P) \subseteq N(P)$. Moreover, since $S_P $ crosses $ P$ horizontally and all elements of $ N(P)$ cross $ P$ vertically, by \Cref{prop:horizont-and-vertical-crossing-intersect}, $ S_P $ intersects all elements of $ N(P)$. Moreover, if $ A\in N(P)$, we have 
	$$
	\{(x,y): x = \rset{S_P}\} = \{(x,y): x = \rset{P^{\uparrow}}\} \subseteq \Ter{A}. 
	$$
	Finally,
	$$
	\lset{A} \leq \lset{P^{\uparrow}} = \lset{S_P} < \rset{A} \leq \rset{\mathcal F} < \rset{S_P},
	$$
	and 
	$$
	\bset{A} < \bset{P^{\uparrow}} = \bset{S_P} < \tset{S_P} = \tset{P^{\uparrow}} < \tset{A}. 
	$$
	So, $ A \adj N(P)$ for all $ A \in N(P)$, and this completes the proof. 
\end{proof}


\end{document}